\documentclass[11pt]{amsart}
\usepackage{mathdots, url}
\usepackage{graphicx}
\usepackage{amsopn}
\usepackage{amsmath, amsfonts, amssymb, amscd, amsthm}
\usepackage{yhmath}
\usepackage{setspace}
\usepackage{mathtools}
\usepackage[all]{xy}
\usepackage{verbatim}
\usepackage{enumerate}
\usepackage[usenames]{color}
\usepackage{varwidth}
\allowdisplaybreaks

\newtheorem{theorem}{Theorem}
\newtheorem*{theorem*}{Theorem}

\newtheorem{lemma}{Lemma}

\newtheorem{proposition}{Proposition}

\theoremstyle{definition}
\newtheorem{definition}{Definition}

\theoremstyle{remark}

\numberwithin{equation}{section}
\allowdisplaybreaks

\title[On reducibility of induced representations]{On reducibility of induced representations of odd unitary groups: the depth zero case}
\author{Subha Sandeep Repaka}
\address{Department of Mathematics, SRM University- AP, Mangalagiri-Mandal, Neeru Konda, Amaravati, Andhra Pradesh 522240}
\email{subhasandeep.r@srmap.edu.in, sandeep.repaka@gmail.com}

\begin{document}
	
	\date{\today}	
	\setcounter{tocdepth}{1}
	\date{\today}
	
	\begin{abstract}
		We study a problem concerning parabolic induction in certain $p$-adic unitary groups. More precisely, for $E/F$ a quadratic extension of $p$-adic fields the associated unitary group $G=\mathrm{U}(n,n+1)$ contains a parabolic subgroup $P$ with Levi component $L$ isomorphic to $\mathrm{GL}_n(E) \times \mathrm{U}_1(E)$. Let $\pi$ be an irreducible supercuspidal representation of $L$ of depth zero. We use Hecke algebra methods to determine when the parabolically induced representation $\iota_P^G \pi$ is reducible.
	\end{abstract}

	\maketitle
	\tableofcontents
	
\section{Introduction}

In this paper we solve a similar problem as the one which we did in \cite{sandeep}. In \cite{sandeep}, we solved the problem for $\mathrm{U}(n,n)$ over non- Archimedean local fields where as in this paper we are solving the same problem for $\mathrm{U}(n,n+1)$ over non- Archimedean local fields. Refer to the section 1 in \cite{sandeep} for a better understanding of what we are doing in this paper. All the representations in this paper are smooth and complex representations. \par

Let $G=\mathrm{U}(n,n+1)$ be the odd unitary group over non- Archimedean local field $E$ and $\pi$ is an irreducible supercuspidal  depth zero representation of the Siegel Levi component $L \cong \mathrm{GL}_n(E) \times \mathrm{U}_1(E)$ of the Siegel parabolic subgroup $P$ of $G$. The terms $P, L, \pi, \mathrm{U}(n,n+1)$ are described in much detail later in the paper. We use Hecke algebra methods to determine when the parabolically induced representation $\iota_P^G \pi$ is reducible. Harish-Chandra tells us to  look not at an individual $\iota_P^G \pi$ but at the family $\iota_P^G (\pi \nu)$ as $\nu$ varies through the unramfied characters of $L \cong {\rm GL}_n(E) \times \mathrm{U}_1(E)$. The unramified characters of $L$ and the functor $\iota_{P}^{G}$ are also described in greater detail later in the paper.\par  

Before going any further, let us describe how $\mathrm{U}(n,n+1)$ over  non-Archimedean local fields looks like. Let $E/F$ be a quadratic Galois extension of non-Archimedean local fields where char $F \neq 2$. Write $-$ for the non-trivial element of $\mathrm{Gal}(E/F)$. The group $G=\mathrm{U}(n,n+1)$ is given by

\[
\mathrm{U}(n,n+1)= \{g \in \mathrm{GL}_{2n+1}(E) \mid {^t}{\overline g}Jg=J\}
\]
for $J=
\begin{bmatrix}
0 &0 &Id_{n}\\
0& 1 & 0 \\
Id_{n} & 0 & 0
\end {bmatrix}$ where each block is of size $n$ and for $g=(g_{ij})$ we write $\overline{g}= (\overline{g}_{ij})$. We write $\mathfrak{O}_E$ and $\mathfrak{O}_F$ for the ring of integers in $E$ and $F$ respectively. Similarly, $\mathbf{p}_E$ and $\mathbf{p}_F$ denote the maximal ideals in $\mathfrak{O}_E$ and $\mathfrak{O}_F$ and $k_E=\mathfrak{O}_E/\mathbf{p}_E$ and $k_F=\mathfrak{O}_F/\mathbf{p}_F$ denote the residue class fields of $\mathfrak{O}_E$ and $\mathfrak{O}_F$. Let $\left|k_F\right|= q =p^r$ for some odd prime $p$ and some integer $r \geqslant 1$.\par  

There are two kinds of extensions of $E$ over $F$. One is the unramified extension and the other one is the ramified extension. In the unramified case, we can choose uniformizers $\varpi_E,\varpi_F$ in $E, F$ such that $\varpi_E=\varpi_F$ so that we have $[k_E:k_F]=2 , \mathrm{Gal}(k_E/k_F) \cong \mathrm{Gal}(E/F)$. As $\varpi_E=\varpi_F$, so $\overline{\varpi}_E=\varpi_E$  since  $\varpi_F \in F$. As $k_F=\mathbb{F}_q$,  so $k_E= \mathbb{F}_{q^2}$ in this case. In the ramified case, we can choose uniformizers $\varpi_E,\varpi_F$ in $E, F$ such that $\varpi_E^2=\varpi_F$ so that we have $[k_E:k_F]=1,\mathrm{Gal}(k_E/k_F)=1$. As $\varpi_E^2=\varpi_F$, we can further choose $\varpi_E$ such that $\overline{\varpi}_E= -\varpi_E$. As $k_F=\mathbb{F}_q$, so $k_E= \mathbb{F}_{q}$ in this case. \par

We write $P$  for the Siegel parabolic subgroup of $G$. Write $L$ for the Siegel Levi component of $P$ and $U$ for the unipotent radical of $P$. Thus $P= L \ltimes U$ with  

\[
L= \Bigg\lbrace \begin{bmatrix}
a & 0 & 0\\
0 &\lambda &0\\
0 &0 &{^t}{\overline {a}}{^{-1}}
\end{bmatrix} \mid a \in \mathrm{GL}_n(E), \lambda \in E^{\times}, \lambda\overline{\lambda}=1 \Bigg \rbrace
\]

and

\[
U= \Bigg \lbrace \begin{bmatrix}
Id_n &u &X \\
0 &1 &-{}^t\overline{u}\\
0 &0 &Id_n
\end{bmatrix} \mid  X \in \mathrm{M}_n(E), u \in \mathrm{M}_{n \times 1}(E), X + {^t}{\overline{X}+u{}^t\overline{u}}=0 \Bigg \rbrace.
\]

Note that $L \cong \mathrm{GL}_n(E) \times \mathrm{U}_1(E)$ and $\mathrm{U}_1(E) \cong \mathrm{U}_1(\mathfrak{O}_E)$. Let $\overline P= L \ltimes \overline U$ be the $L$-opposite of $P$ where 
\[
\overline{U}= \Bigg \lbrace \begin{bmatrix}
Id_n &0 &0 \\
-{}^t\overline{u} &1 &0\\
X &u &Id_n
\end{bmatrix} \mid  X \in \mathrm{M}_n(E), u \in \mathrm{M}_{n \times 1}(E), X + {^t}{\overline{X}+u{}^t\overline{u}}=0 \Bigg \rbrace.
\]

Let $K_0=\mathrm{GL}_n(\mathfrak{O}_E)$ and $K_1=Id_n+\varpi_E \mathrm{M}_n(\mathfrak{O}_E)$. Note $K_1=Id_n+\varpi_E \mathrm{M}_n(\mathfrak{O}_E)$ is the kernel of the surjective group homomorphism \[(g_{ij}) \longrightarrow (g_{ij} + \mathbf{p_E}) \colon \mathrm{GL}_n(\mathfrak{O}_E) \longrightarrow \mathrm{GL}_n(k_E)\] \par

As $\pi$ is a depth zero representation of $L \cong \mathrm{GL}_n(E) \times \mathrm{U}_1(E)$. So $\pi = \lambda \chi$ where $\lambda$ is a depth zero representation of $\mathrm{GL}_n(E)$ and $\chi$ is a depth zero character of $\mathrm{U}_1(E)$. We say $\pi$ is a depth zero representation of the Siegel Levi component $L$ of $P$ if $\lambda^{K_1} \neq 0$ and $\chi|_{\mathrm{U}_1(1+p_E)}=1$.\par

 Let $(\rho, V)$ be a smooth representation of the group $H$ which is a subgroup of $K$. The smoothly induced representation from $H$ to $K$ is denoted by $Ind_H^K(\rho, V)$ or $Ind_H^K(\rho)$. Let us denote $c$-$Ind_H^K(\rho, V)$ or  $c$-$Ind_P^G(\rho)$ for smoothly induced compact induced representation from $H$ to $K$.\par

The normalized induced representation from $P$ to $G$ is denoted by $\iota_P^G(\rho, V)$ or $\iota_P^G(\rho)$ where $\iota_P^G(\rho)=Ind_P^G(\rho \otimes \delta_P^{1/2})$, $\delta_P$ is a character of $P$ defined as $\delta_P(p)=\|det(Ad \, p)|_{\mathrm{Lie} \, U}\|_F$ for $p \in P$ and $\mathrm{Lie} \, U$ is the Lie-algebra of $U$. We work with normalized induced representations rather than induced representations in this paper as results look more appealing as for example it commutes with taking duals.\par 

Write $L^{\circ}$ for the smallest subgroup of $L$ containing the compact open subgroups of $L$. We say a character $\nu \colon L \longrightarrow \mathbb{C}^{\times}$ is unramified if $\nu|_{L^{\circ}}=1$. Observe that if $\nu$ is an unramified character of $L$ then $\nu= \nu^{'}\beta$ where $\nu^{'}$ is an unramified character of $\mathrm{GL}_n(E)$ and $\beta$ is an unramified character of $\mathrm{U}_1(E)$. But as $\mathrm{U}_1(E)=\mathrm{U}_1(\mathfrak{O}_E)$, so $\beta$ is trivial. Hence, $\nu$ can be viewed as an unramified character of $\mathrm{GL}_n(E)$. Let the group of unramified characters of $L$ be denoted by $\mathrm{X}_{nr}(L)$.\par
 
\subsection{Question}
The question we answer in this paper is, given $\pi$ an irreducible supercuspidal representation of $L$ of depth zero, we look at the family of representations $\iota_P^G(\pi\nu)$  for $\nu \in \chi_{nr}(L)$ and we want to determine the set of such $\nu$ for which this induced representation is reducible for both ramified and unramified extensions. By general theory, this is a finite set.\par

Recall that $\pi= \lambda \chi$ where $\lambda$ is an irreducible supercuspidal depth zero representation of $\mathrm{GL}_n(E)$ and $\chi$ is a supercuspidal depthzero character of $\mathrm{U}_1(E)$. Now $\lambda|_{K_0}$ contains an irreducible representation $\tau$ of $K_0$ such that $\tau|_{K_1}$ is trivial. So $\tau$ can be viewed as an irreducible representation of $K_0/K_1\cong \mathrm{GL}_n(k_E)$  inflated to $K_0=\mathrm{GL}_n(\mathfrak{O}_E)$. The representation $\tau$ is cuspidal by (a very special case of) A.1 Appendix \cite{MR1235019}. Set $\rho_0= \tau \chi$ which is a cuspidal representation of $K_0 \times \mathrm{U}_1(\mathfrak{O}_E)$. Further, we can view $\rho_0=\tau \chi$ as a cuspidal representation of $\mathrm{GL}_n(k_E) \times \mathrm{U}_1(k_E)$ inflated to $K_0 \times \mathrm{U}_1(\mathfrak{O}_E)$.\par

By the work of Green \cite{green} or as a very special case of the Deligne-Lusztig construction, irreducible cuspidal representations of $\mathrm{GL}_n(k_E)$ are parametrized by the regular characters of degree $n$ extensions of $k_E$. We write $\tau_\theta$ for the irreducible cuspidal representation $\tau$ that corresponds to a regular character $\theta$. Let $l/k_E$ be a field extension of degree $n$. We set $\Gamma = {\rm Gal}(l/k_E)$.

Let 
\[
(l^\times)^\vee = {\rm Hom}(l^\times, \mathbb{C}^\times).
\]
Clearly, $\Gamma$ acts on $(l^\times)^\vee$ via
\[
\theta^\gamma (x) = \theta({}^\gamma x), \quad \theta \in (l^\times)^\vee, \,\,\gamma \in \Gamma, \,\,x \in l^\times.
\]     
We write $(l^\times)^\vee_{{\rm reg}}$ for the group of regular characters of $l^\times$ with respect to this action, that is, 
characters $\theta$ such that ${\rm Stab}_\Gamma (\theta)= \{ 1 \}$. 
We also write $l^\times_{\rm reg}$ for the regular elements in $l^\times$, that is, elements $x$ such that 
${\rm Stab}_\Gamma (x)= \{ 1 \}$. 
The set of $\Gamma$-orbits on  
$(l^\times)^\vee_{{\rm reg}}$ is then in canonical bijection with the set ${\rm Irr}_{\rm cusp} \, {\rm GL}_n(k_E)$ of equivalence 
classes of irreducible cuspidal representations of ${\rm GL}_n(k_E)$:
\begin{align*} 
\Gamma \backslash (l^\times)^\vee_{{\rm reg}}  \,\,\,     
&\longleftrightarrow \,\,\,  {\rm Irr}_{\rm cusp} \, {\rm GL}_n(k_E)  \\
\theta &\longleftrightarrow \tau_\theta.
\end{align*} 
The bijection is specified by a character relation         
\[
\tau_\theta (x) = c  \sum_{\gamma \in \Gamma}    \theta^\gamma(x), \quad x \in l^\times_{\rm reg}, 
\]
for a certain constant $c$ that is independent of $\theta$ and $x$. We denote $\tau$ by $\tau_{\theta}$.\par

Note that we have $k_E=\mathbb{F}_{q^2}$. So $l= \mathbb{F}_{q^{2n}}$.\par

As $\Gamma=\mathrm{Gal}(l/k_E)$, $\Gamma$ is generated by the Frobenius map $\Phi$ given by $\Phi(\lambda)=\lambda^{q^2}$ for $\lambda \in l$. Note that here $\theta^{\Phi}= \theta^{q^2}$. Also observe that $\Phi^n(\lambda)= \lambda^{q^{2n}}=\lambda$ (since $l^{\times}$ is a cyclic group of order $q^{2n}-1$) $\Longrightarrow \Phi^n= 1$. \par

Note that for two regular characters $\theta$ and $\theta'$ we have $\tau_{\theta}\simeq\tau_{\theta'} \Longleftrightarrow$ there exists $\gamma \in \Gamma$ such that $\theta^{\gamma}= \theta'$. \par

We now define the Siegel parahoric subgroup $\mathfrak{P}$ of $G$ which is given by:

\begin{center}
	$\mathfrak{P} =
	\begin{bmatrix}
	\mathrm{GL}_n(\mathfrak{O}_E)    & \mathrm{M}_{n \times 1}(\mathfrak{O}_E) & \mathrm{M}_n(\mathfrak{O}_E) \\
	\mathrm{M}_{1 \times n}(\mathbf{p}_E) & \mathrm{U}_1(\mathfrak{O}_E) & \mathrm{M}_{1 \times n}(\mathfrak{O}_E)\\
	\mathrm{M}_n(\mathbf{p}_E)  & \mathrm{M}_{n \times 1}(\mathbf{p}_E) & \mathrm{GL}_n(\mathfrak{O}_E) 
	\end{bmatrix}
	\bigcap \mathrm{U}(n,n+1).$\\
\end {center}

We have $\mathfrak{P}= (\mathfrak{P}\cap\overline{U})(\mathfrak{P}\cap L)(\mathfrak{P}\cap U) $(Iwahori factorization of $\mathfrak{P}$). Let us denote $(\mathfrak{P}\cap\overline{U})$ by $\mathfrak{P}_{-}$, $(\mathfrak{P}\cap U)$ by $\mathfrak{P}_{+}$, $(\mathfrak{P}\cap L)$ by $\mathfrak{P}_{0}$. Thus

	\begin{center}
		$ \mathfrak{P}_{0}=
		\Bigg\lbrace
		\begin{bmatrix}
		a & 0 &0\\
		0& \lambda & 0\\ 
		0 & 0&{^t}{\overline{a}}{^{-1}}
		\end{bmatrix}\mid a \in \mathrm{GL}_n(\mathfrak{O}_E), \lambda \in \mathfrak{O}_E^{\times}, \lambda\overline{\lambda}=1 
		\Bigg \rbrace$,
		\end {center}
		
	\begin{center}
		$\mathfrak{P}_{+}= \Bigg \lbrace \begin{bmatrix}
			Id_n &u &X \\
			0 &1 &-{}^t\overline{u}\\
			0 &0 &Id_n
		\end{bmatrix} \mid  X \in \mathrm{M}_n(\mathfrak{O}_E), u \in \mathrm{M}_{n \times 1}(\mathfrak{O}_E), X + {^t}{\overline{X}+u{}^t\overline{u}}=0 \Bigg \rbrace,$
		\end{center}
			
    \begin{center}
    $\mathfrak{P}_{-}= \Bigg \lbrace \begin{bmatrix}
	Id_n &0 &0 \\
	-{}^t\overline{u} &1 &0\\
	X &u &Id_n
	\end{bmatrix} \mid  X \in \mathrm{M}_n(\mathbf{p}_E), u \in \mathrm{M}_{n \times 1}(\mathbf{p}_E), X + {^t}{\overline{X}+u{}^t\overline{u}}=0 \Bigg \rbrace.$
	\end{center}

By Iwahori factorization of $\mathfrak{P}$ we have $\mathfrak{P}= (\mathfrak{P}\cap\overline{U})(\mathfrak{P}\cap L)(\mathfrak{P}\cap U)=\mathfrak{P}_{-}\mathfrak{P}_{0}\mathfrak{P}_{+}$. As $\rho_0$ is a representation of $K_0 \times \mathrm{U}_1(\mathfrak{O}_E)$, it can also be viewed as a representation of $\mathfrak{P}_{0}$. This is because $\mathfrak{P}_{0}\cong K_0 \times \mathrm{U}_1(\mathfrak{O}_E)$.\par

Let $Z(L)$ denote the center of $L$. Hence 

\[
Z(L)=\Bigg\lbrace
\begin{bmatrix}
aId_n  & 0 &0\\
0 & \lambda &0\\
0    & 0 & \overline{a}^{-1}Id_n \\
\end{bmatrix} \mid a \in E^{\times}, \lambda \in E^{\times},  \lambda\overline{\lambda}=1 \Bigg \rbrace.
\]\par

Let us set 

\[
\zeta=
\begin{bmatrix}
\varpi_E Id_n & 0 &0\\
0 &1 &0\\
0 &0 & {\overline\varpi_E}^{-1}Id_n
\end{bmatrix}.\]\par

Note that  $Z(L)\mathfrak{P}_0= \coprod_{n \in \mathbb{Z}} \mathfrak{P}_0 \zeta^n$, so we can extend $\rho_0$ to a representation $\widetilde{\rho_0}$ of $Z(L)\mathfrak{P}_0$ via $\widetilde{\rho_0}(\zeta^k j)=\rho_0(j)$ for $j \in \mathfrak{P}_0, k \in \mathbb{Z}$. By standard Mackey theory arguments, we show in the paper that $\pi$= $c$-$Ind_{Z(L)\mathfrak{P}_0}^L \widetilde{\rho_0}$ is a smooth irreducible supercuspidal depth zero representation of $L$. Also note that any arbitrary depth zero irreducible supercuspidal representation of $L$ is an unramified twist of $\pi$. To that end, we will answer the question which we posed earlier in this paper and prove the following result.

\begin{theorem} \label{the_1}
	Let $G= \mathrm{U}(n,n+1)$. Let $P$ be the Siegel parabolic subgroup of $G$ and $L$ be the Siegel Levi component of $P$. Let $\pi$= $c$-$Ind_{Z(L)\mathfrak{P}_0}^L \widetilde{\rho_0}$ be a smooth irreducible supercuspidal depth zero representation of $L \cong \mathrm{GL}_n(E) \times \mathrm{U}_1(E)$ where $\widetilde{\rho_0}(\zeta^k j)=\rho_0(j)$ for $j \in \mathfrak{P}_0, k \in \mathbb{Z}$ and  $\rho_0=\tau_{\theta}$ for some regular character $\theta$ of $l^{\times}$ with $[l:k_E]=n$ and $|k_F|= q$. Consider the family $\iota_P^G(\pi\nu)$ for $\nu \in \mathrm{X}_{nr}(L)$.
	\begin{enumerate}
	\item For $E/F$ unramified, $\iota_P^G(\pi \nu)$ is reducible $\Longleftrightarrow n$ is odd, $\theta^{q^{n+1}}= \theta^{-q}$ and $\nu(\zeta) \in \{q^{n},q^{-n},-1\}$.
	\item For $E/F$ ramified, $\iota_P^G(\pi \nu)$ is reducible $\Longleftrightarrow n$ is even, $\theta^{q^{n/2}}=\theta^{-1}$ and $\nu(\zeta) \in \{q^{n/2},q^{-n/2},-1\}$.
\end{enumerate}
\end{theorem}

In \cite{MR1266747}, Goldberg computes the reducibility points of $\iota_{P}^{G}(\pi)$ by computing the poles of certain $L$-functions attached to the representations of $\mathrm{GL}_n(E)$. Note however that the base field $F$ is assumed to be of characteristic $0$ in \cite{MR1266747}, whereas we assumed characteristic of $F \neq 2$. In \cite{MR1266747} no restriction on the depth of the representation $\pi$ is there, while in this paper we have assumed depth of the representation $\pi$ to be of zero. The final results obtained in \cite{MR1266747} are in terms of matrix coefficents of $\pi$ whereas our results are in terms of the unramified characters of $L$.\par

In \cite{heiermann_2011} and \cite{heiermann_2017}, Heiermann computed the structure of the Hecke algebras which we look at and makes a connection with Langlands parameters. But his results are not explicit. They do not give the precise values of the parameters in the relevant Hecke algebras.\par

In \cite{Stevens}, Stevens and Lust have calculated the parameters of the affine Hecke algebras for all the classical groups, so in particular they have also calculated the parameters of affine Hecke algebras of odd unitary groups in the depth zero setting for both ramified and unramified extensions. However, the approach taken by them is quite different from our approach.

$\mathbf{Acknowledgments}$: The author wishes to thank Alan Roche from University of Oklahoma, USA for suggesting the problem studied in this work and for many discussions and insights.

\section{Preliminaries}
\subsection{Bernstein Decomposition}
Let $G$ be the $F$-rational points of a reductive algebraic group defined over a non-Archimedean local field $F$. Let $(\pi, V)$ be an irreducible smooth representation of $G$. According to Theorem 3.3 in \cite{MR1486141}, there exists unique conjugacy class of cuspidal pairs $(L, \sigma)$ with the property that $\pi$ is isomorphic to a composition factor of $\iota_P^G\sigma$ for some parabolic subgroup $P$ of $G$. We call this conjugacy class of cuspidal pairs, the cuspidal support of $(\pi, V)$.\par

Given two cuspidal supports $(L_1, \sigma_1)$ and $(L_2, \sigma_2)$ of $(\pi, V)$, we say they are inertially equivalent if there exists $g \in G$ and $\chi \in \mathrm{X}_{nr}(L_2)$ such that $L_2= L_1^g$ and $\sigma_1^g\simeq \sigma_2\otimes\chi$. We write $[L, \sigma]_G$ for the inertial equivalence class or inertial support of $(\pi, V)$. Let $\mathfrak{B}(G)$ denote the set of inertial equivalence classes $[L, \sigma]_G$.\par

Let $\mathfrak{R}(G)$ denote the category of smooth representations of $G$. Let ${\mathfrak{R}}^s(G)$ be the full sub-category of smooth representations of $G$ with the property that $(\pi, V) \in ob({\mathfrak{R}}^s(G))\Longleftrightarrow$  every irreducible sub-quotient of $\pi$ has inertial support $s=[L, \sigma]_G$.\par

We can now state the Bernstein decomposition:

\[\mathfrak{R}(G)=\prod_{s \in \mathfrak{B}(G)}{\mathfrak{R}}^s(G).\]\par

\subsection{Types}

Let $G$ be the $F$-rational points of a reductive algebraic group defined over a non-Archimedean local field $F$. Let $K$ be a compact open subgroup of $G$. Let $(\rho, W)$ be an irreducible smooth representation of $K$ and $(\pi, V)$  be a  smooth representation of $G$. Let $V^{\rho}$ be the $\rho$-isotopic subspace of $V$.

\[
V^{\rho}= \sum\limits_{W'}W'
\] where the sum is over all $W'$ such that $(\pi|_K, W') \simeq (\rho, W)$.\par

Let $\mathcal{H}(G)$ be the space of all locally constant compactly supported functions $f \colon G \to \mathbb{C}$. This is a $\mathbb{C}$- algebra under convolution $*$. So for elements $f,g \in \mathcal{H}(G)$ we have 

\[(f * g)(x)= \int_G f(y)g(y^{-1}x)d\mu(y).\]

Here we have fixed a Haar measure $\mu$ on $G$. Let $(\pi, V)$ be a representation of $G$. Then  $\mathcal{H}(G)$ acts on $V$ via

\[hv=\int_G h(x)\pi(x)vd\mu(x)\] for $h \in \mathcal{H}(G), v \in V$. Let $e_\rho$ be the element in $\mathcal{H}(G)$ with support $K$ such that

\begin{center}
	$e_\rho(x)= \frac{\text{dim}\rho}{\mu(K)}tr_W(\rho(x^{-1})), x \in K.$\\
\end{center}

We have $e_\rho * e_\rho = e_\rho $ and $e_\rho V = V^{\rho}$ for any smooth representation $(\pi, V)$ of $G$. Let $\mathfrak{R}_\rho(G)$ be the full sub-category of $\mathfrak{R}(G)$ consisting of all representations $(\pi, V)$ where $V$ is generated by $\rho$-isotopic vectors. So $(\pi, V) \in \mathfrak{R}_\rho(G) \Longleftrightarrow  V = \mathcal{H}(G) * e_\rho V$. We now state the definition of a type.

\begin{definition}
	Let $ s \in \mathfrak{B}(G)$. We say that $(K, \rho)$ is an $s$-type in $G$ if $\mathfrak{R}_\rho(G)=
	\mathfrak{R}^s(G)$.
\end{definition}

\subsection{Hecke algebras}

Let $G$ be the $F$-rational points of a reductive algebraic group defined over a non-Archimedean local field $F$. Let $K$ be a compact open subgroup of $G$. Let $(\rho, W)$ be an irreducible smooth representation of $K$. Here we introduce the Hecke algebra $\mathcal{H}(G,\rho)$.

\[
\mathcal{H}(G,\rho)= \left\lbrace f \colon G \to End_{\mathbb{C}}(\rho^{\vee}) \; \middle|  \;
\begin{varwidth}{\linewidth}
supp($f$) is compact and \\
$f(k_1gk_2)= \rho^{\vee}(k_1)f(g)\rho^{\vee}(k_2)$\\
where $k_1,k_2 \in K, g \in G$
\end{varwidth}
\right \rbrace.
\]\par

Then $\mathcal{H}(G,\rho)$ is a $\mathbb{C}$-algebra with multiplication given by convolution $*$ with respect to some fixed Haar measure $\mu$ on $G$. So for elements $f,g \in \mathcal{H}(G, \rho)$ we have 
\[(f * g)(x)= \int_G f(y)g(y^{-1}x)d\mu(y).\] \par

The importance of types is seen from the following result. Let $\pi$ be a smooth representation in $\mathfrak{R}^{s}(G)$. Let $\mathcal{H}(G, \rho)- Mod$ denote the category of $\mathcal{H}(G, \rho)$-modules. If $(K, \rho)$ is an $s$-type then $ m_G \colon \mathfrak{R}^{s}(G) \longrightarrow \mathcal{H}(G, \rho)- Mod$  given by $m_G(\pi)= \mathrm{Hom}_K(\rho,\pi)$ is an equivalence of categories.\par

\subsection{Covers}

Let $G$ be the $F$-rational points of a reductive algebraic group defined over a non-Archimedean local field $F$. Let $K$ be a compact open subgroup of $G$. Let $(\rho, W)$ be an irreducible representation of $K$. Then we say $(K, \rho)$ is decomposed with respect to $(L,P)$ if the following hold:

\begin{enumerate}
	\item $K=(K\cap \overline U)( K \cap L)(K \cap U)$.
	\item $(K\cap \overline U),(K \cap U) \leqslant \text{ker}\rho$.
\end{enumerate}

Suppose $(K, \rho)$ is decomposed with respect  to $(L,P)$. We set $K_L=K \cap L$ and $\rho_L= \rho|_{K_L}$. We say an element $g \in G$ intertwines $\rho$ if $\mathrm{Hom}_{K^g \cap K} (\rho^g, \rho) \neq 0 $. Let $\mathfrak{I}_G(\rho)= \{ x \in G \mid x \,\text{intertwines} \, \rho \}$. We have the Hecke algebras $\mathcal{H}(G,\rho)$ and $\mathcal{H}(L,\rho_L)$. We write

\[
\mathcal{H}(G, \rho)_L =\{f \in \mathcal{H}(G, \rho) \mid \text{supp} (f) \subseteq KLK \}. 
\]\par

We recall some results and constructions from pages 606-612 in \cite{MR1643417}. These allow us to transfer questions about parabolic induction into questions concerning the module theory of appropriate Hecke algebras. 

\begin{proposition}[Bushnell and Kutzko, Proposition 6.3  \cite{MR1643417}]\label{pro_1}
Let $(K, \rho)$ decompose with respect to $(L,P)$ .Then

\begin{enumerate}
	\item $\rho_L$ is irreducible.
	\item $\mathfrak{I}_L(\rho_L)= \mathfrak{I}_G(\rho) \cap L$.
	\item There is an embedding  $T \colon \mathcal{H}(L, \rho_L) \longrightarrow \mathcal{H}(G, \rho)$ such that if $f \in \mathcal{H}(L, \rho_L)$ has support $K_LzK_L$ for some $ z\in L$, then $T(f)$ has support $KzK$.
	\item The map $T$ induces an isomorphism of vector spaces:
	
	\begin{center}
		$\mathcal{H}(L, \rho_L) \xrightarrow{\simeq} \mathcal{H}(G, \rho)_L.$\\
	\end{center}
	
	\end {enumerate}
	
\end{proposition}

\begin{definition}
	
	An element $z \in L$ is called $(K,P)$-positive element if:
	
	\begin{enumerate}
		\item $z(K\cap \overline{U})z^{-1} \subseteq K \cap \overline{U}.$
		\item $z^{-1}(K \cap U)z \subseteq K \cap U.$
		\end{enumerate}
	\end{definition}
	
	\begin{definition}
		An element $z \in L$ is called strongly $(K,P)$-positive element if:
		
		\begin{enumerate}
			\item $z$ is $(K,P)$ positive.
			\item $z$ lies in center of $L$.
			\item For any compact open subgroups $K$ and $K'$ of $U$ there exists $m\geqslant 0$ and $m \in \mathbb{Z}$ such that $z^mKz^{-m} \subseteq K'$.
			\item For any compact open subgroups $K$ and $K'$ of $\overline{U}$ there exists $m\geqslant 0$ and $m \in \mathbb{Z}$ such that $z^{-m}Kz \subseteq K'$. 
		\end{enumerate}

	\end{definition}\par

	\begin{proposition}[Bushnell and Kutzko, Lemma 6.14 \cite{MR1643417}, Proposition 7.1, \cite{MR1643417}]\label{pro_2}
		
		Strongly $(K,P)$-positive elements exist and given a strongly positive element $z \in L$ , there exists a unique function $\phi_z \in \mathcal{H}(L, \rho_L)$ with support $K_LzK_L$ such that $\phi_z(z)$ is identity function in $End_{\mathbb{C}}(\rho_L^{\vee})$.
	\end{proposition}
	
	\[
	\mathcal{H}^{+}(L,\rho_L)= \left\lbrace f \colon L \to End_{\mathbb{C}}(\rho_L^{\vee}) \; \middle|  \;
	\begin{varwidth}{\linewidth}
	supp($f$) is compact and consists\\ 
	of strongly $(K,P)$-positive elements \\ 
	and $f(k_1lk_2)= \rho_L^{\vee}(k_1)f(l)\rho_L^{\vee}(k_2)$\\
	where $k_1,k_2 \in K_L, l \in L$
	\end{varwidth}
	\right \rbrace.
	\]
	
	The isomorphism of vector spaces $T \colon \mathcal{H}(L,\rho_L) \longrightarrow \mathcal{H}(G,\rho)_L$ restricts to an embedding of algebras:

	\begin{center}
		$T^+ \colon \mathcal{H}^{+}(L,\rho_L) \longrightarrow \mathcal{H}(G,\rho)_L \hookrightarrow \mathcal{H}(G,\rho)$.\\
	\end{center}\par

	\begin{proposition}[Bushnell and Kutzko, Theorem 7.2.i \cite{MR1643417}]\label{pro_3}
		The embedding $T^+$ extends to an embedding of algebras $t \colon \mathcal{H}(L,\rho_L)\longrightarrow \mathcal{H}(G,\rho) \Longleftrightarrow T^+(\phi_z)$ is invertible for some strongly $(K,P)$-positive element $z$, where  $\phi_z \in \mathcal{H}(L,\rho_L)$ has support $K_LzK_L$ with $\phi_z(z)=1$.
	\end{proposition}
	
	\begin{definition}\label{def}
		Let $L$ be a proper Levi subgroup of $G$. Let $K_L$ be a compact open subgroup of $L$ and $\rho_L$ be an irreducible smooth representation of $K_L$. Let $K$ be a compact open subgroup of $G$ and $\rho$ be an irreducible, smooth representation of $K$. Then we say $(K, \rho)$ is a $G$-cover of $(K_L, \rho_L)$
		if
		
		\begin{enumerate}
		\item The pair $(K, \rho)$ is decomposed with respect to $(L,P)$ for every parabolic subgroup $P$ of $G$ with Levi component $L$.
		\item $K \cap L= K_L$  and $\rho|_L \simeq \rho_L$.
		\item The embedding $T^+ \colon \mathcal{H}^{+}(L,\rho_L) \longrightarrow \mathcal{H}(G,\rho)$ extends to an embedding of algebras $ t \colon \mathcal{H}(L,\rho_L) \longrightarrow \mathcal{H}(G,\rho)$.
	\end{enumerate}
\end{definition}\par

\begin{proposition}[Bushnell and Kutzko, Theorem 8.3 \cite{MR1643417}]\label{Types_Covers}
	Let $s_L=[L,\pi]_L \in \mathfrak{B}(L)$ and $s= [L,\pi]_G \in \mathfrak{B}(G)$ . Say $(K_L,\rho_L)$ is an $s_L$-type and $(K,\rho)$ is a $G$-cover of $(K_L,s_L)$. Then $(K,\rho)$ is an $s$-type.
\end{proposition}\par

Note that in this paper $K=\mathfrak{P}, K_L=K \cap L= \mathfrak{P} \cap L= \mathfrak{P}_0$ and $\rho_L= \rho_0$. Also note that in this paper, $\rho$ is defined as $\rho(p)=\rho_0(p_0)$ for $p \in \mathfrak{P}$ where by Iwahori Factorization $p=p_{+}p_{0}p_{-}, p_{+}\in \mathfrak{P} \cap U, p_0 \in \mathfrak{P}_0, p_{-} \in \mathfrak{P} \cup \overline{U}$. Observe that from definition \ref{def}, we can conclude that $(\mathfrak{P},\rho)$ is a cover of $(\mathfrak{P}_0,\rho_0)$. Also observe that as $(\mathfrak{P_0},\rho_0)$ is $s_L$-type and as $(\mathfrak{P},\rho)$ is a cover of $(\mathfrak{P}_0,\rho_0)$, so from proposition \ref{Types_Covers}, it follows that  $(\mathfrak{P},\rho)$ is a $s$-type.\par

Recall the categories $\mathfrak{R}^{s_L}(L), \mathfrak{R}^{s}(G)$ where $s_L=[L,\pi]_L$ and $s= [L,\pi]_G$. Note that $\pi \nu$ lies in the category $\mathfrak{R}^{s_L}(L)$ and $\iota_P^G(\pi \nu)$ lies in $\mathfrak{R}^{s}(G)$.\par

Note that $\mathcal{H}(G,\rho)-Mod$ is the category of $\mathcal{H}(G,\rho)$-modules and $\mathcal{H}(L,\rho_L)-Mod$ be the category of $\mathcal{H}(L,\rho_L)$-modules.\par

The functor $\iota_P^G$ was defined earlier. The functor $m_L \colon \mathfrak{R}^{s_L}(L) \longrightarrow  \mathcal{H}(L,\rho_L)-Mod$ is given by $m_L(\pi \nu)= \mathrm{Hom}_{K_L}(\rho_L, \pi \nu)$. The representation $\pi \nu \in \mathfrak{R}^{s_L}(L)$  being irreducible, it corresponds to a simple $\mathcal{H}(L,\rho_0)$-module under the functor $m_L$. Let $f \in m_L(\pi \nu), \gamma \in \mathcal{H}(L,\rho_0)$ and $w \in V$. The action of $\mathcal{H}(L,\rho_0)$ on $m_L(\pi \nu)$ is given by $(\gamma. f)(w)= \int_L \pi(l) \nu(l)f(\gamma^{\vee}(l^{-1})w)dl$. Here $\gamma^{\vee}$ is defined on $L$ by $\gamma^{\vee}(l^{-1})=\gamma(l)^{\vee}$ for $l \in L$.\par

The functor $m_G \colon \mathfrak{R}^{s}(G) \longrightarrow  \mathcal{H}(G,\rho)-Mod$ is given by: \[m_G(\iota_P^G(\pi \nu))= \mathrm{Hom}_{K}(\rho, \iota_P^G(\pi \nu)).\] \par

Further the functor $(T_P)_* \colon  \mathcal{H}(L,\rho_L)-Mod \longrightarrow \mathcal{H}(G,\rho)-Mod$ is given by, for $M$ an $\mathcal{H}(L, \rho_0)$-module, 

\[(T_P)_{*}(M)= \mathrm{Hom}_{\mathcal{H}(L, \rho_0)}(\mathcal{H}(G, \rho),M)\] where $\mathcal{H}(G, \rho)$ is viewed as a $\mathcal{H}(L, \rho_0)$-module via $T_P$. The action of $\mathcal{H}(G,\rho)$ on $(T_P)_*(M)$ is given by 

\[h'\psi(h_1)=\psi(h_1 h')\] where $\psi \in (T_P)_*(M), h_1, h' \in \mathcal{H}(G, \rho)$.\par
\par

The importance of covers is seen from the following commutative diagram  which we will use in answering the question which we posed earlier in this paper.

\[
\begin{CD}
\mathfrak{R}^{s}(G)    @>m_G>>    \mathcal{H}(G,\rho)-Mod\\
@A\iota_P^GAA                                    @A(T_P)_*AA\\
\mathfrak{R}^{s_L}(L)    @>m_L>>     \mathcal{H}(L,\rho_0)-Mod
\end{CD}
\]\par

Let us denote the set of strongly $(\mathfrak{P},P)$-positive elements by $\mathcal{I}^{+}$. Thus

\[
\mathcal{I}^{+}=\{x \in L \mid x\mathfrak{P}_{+}x^{-1}\subseteq\mathfrak{P}_{+}, x^{-1}\mathfrak{P}_{-}x \subseteq \mathfrak{P}_{-} \}.
\] \par

where $\mathfrak{P}_{+}=\mathfrak{P} \cap U, \mathfrak{P}_{-}= \mathfrak{P} \cap \overline{U}$. Let $V$ be the vector space corresponding to $\rho_0$. We shall show in section 7 that $\mathcal{H}(L, \rho_0) = \mathbb{C}[\alpha, \alpha^{-1}]$ where $\alpha \in \mathcal{H}(L, \rho_0)$ has support $\mathfrak{P}_0 \zeta \mathfrak{P}_0$ and $\alpha(\zeta)=1_{{V}^{\vee}}$. We will also show that $\alpha^n(\zeta^n) = (\alpha(\zeta))^n$ for $n \in \mathbb{Z}$ and $\text{supp}(\alpha^n)= \mathfrak{P}_0 \zeta^n \mathfrak{P}_0= \mathfrak{P}_0 \zeta^n =  \zeta^n \mathfrak{P}_0$ for $n \in \mathbb{Z}$.\\

We have
\[
\mathcal{H}^{+}(L, \rho_0)= \{f \in \mathcal{H}(L, \rho_0) \mid \text{supp}f \subseteq \mathfrak{P}_0\mathcal{I}^{+}\mathfrak{P}_0 \}.
\]\par

Note $\zeta \in \mathcal{I}^{+}$, so $\mathcal{H}^{+}(L, \rho_0)= \mathbb{C}[\alpha]$. The following discussion is taken from pages 612-619 in \cite{MR1643417}. Let $W$ be space of $\rho_0$. Let $f \in \mathcal{H}^{+}(L, \rho_0)$ with support of $f$ being $\mathfrak{P}_0 x\mathfrak{P}_0$ for $x \in  \mathcal{I}^{+}$. The map $F \in \mathcal{H}(G, \rho)$ is supported on $\mathfrak{P} x \mathfrak{P}$ and $f(x)=F(x)$. The algebra embedding 
\[T^{+} \colon \mathcal{H}^{+}(L, \rho_0) \longrightarrow \mathcal{H}(G, \rho)\] is given by $T^{+}(f)=F$, where $F$ is invertible.\par

Recall support of $\alpha \in \mathcal{H}^{+}(L, \rho_0)$ is $\mathfrak{P}_0 \zeta$. Let $T^{+}(\alpha)= \psi$, where $\psi \in \mathcal{H}(G, \rho)$ has support $\mathfrak{P} \zeta \mathfrak{P}$ and $\alpha(\zeta)=\psi(\zeta)=1_{W^{\vee}}$. As $T^{+}(\alpha)= \psi$ is invertible, so from Proposition ~\ref{pro_3} we can conclude that $T^{+}$ extends to an embedding of algebras  
\[t \colon \mathcal{H}(L, \rho_0) \longrightarrow \mathcal{H}(G, \rho).\]

Let $\phi \in \mathcal{H}(L, \rho_0)$  and $m \in \mathbb{N}$ is chosen such that $\alpha^m \phi \in \mathcal{H}^{+}(L, \rho_0)$. The map $t$ is then given by $t(\phi)= \psi^{-m}T^{+}(\alpha^m \phi)$. For  $\phi \in \mathcal{H}(L,\rho_0)$, the map 
\[t_P \colon \mathcal{H}(L,\rho_0) \longrightarrow \mathcal{H}(G, \rho)\] is given by $t_P(\phi)= t(\phi \delta_P)$, where $\phi \delta_P \in \mathcal{H}(L,\rho_0)$ and is the map \[\phi \delta_P \colon L \longrightarrow End_{\mathbb{C}}(\rho_0^{\vee})\] given by $(\phi \delta_P)(l)=\phi(l)\delta_P(l)$ for $l \in L$. As $\alpha \in \mathcal{H}(L, \rho_0)$ we have
\begin{align*}
t_P(\alpha)(\zeta) &= t(\alpha \delta_P)(\zeta)\\
&= T^{+}(\alpha \delta_P)(\zeta)\\
&= \delta_P(\zeta)T^{+}(\alpha)(\zeta)\\
&= \delta_P(\zeta)\psi(\zeta)\\
&= \delta_P(\zeta)1_{W^{\vee}}.
\end{align*}\par

Let $\mathcal{H}(L, \rho_0)$-Mod denote the category of $\mathcal{H}(L, \rho_0)$-modules and  $\mathcal{H}(G, \rho)$-Mod denote the category of $\mathcal{H}(G, \rho)$-modules. The map $t_P$ induces a functor $(t_P)_{*}$ given by

\[(t_P)_{*} \colon \mathcal{H}(L, \rho_0)-\text{Mod} \longrightarrow \mathcal{H}(G, \rho)-\text{Mod}.\] 

For $M$ an $\mathcal{H}(L, \rho_0)$-module,

\[(t_P)_{*}(M)= \mathrm{Hom}_{\mathcal{H}(L, \rho_0)}(\mathcal{H}(G, \rho),M)\] where $\mathcal{H}(G, \rho)$ is viewed as a $\mathcal{H}(L, \rho_0)$-module via $t_P$. The action of $\mathcal{H}(G,\rho)$ on $(t_P)_*(M)$ is given by 

\[h'\psi(h_1)=\psi(h_1 h')\] where $\psi \in (t_P)_*(M), h_1, h' \in \mathcal{H}(G, \rho)$.\par

Let $\tau \in \mathfrak{R}^{[L,\pi]_L}(L)$ then functor $m_L \colon \mathfrak{R}^{[L,\pi]_L}(L) \longrightarrow \mathcal{H}(L,\rho_0)-Mod$ is given by $m_L(\tau)= \mathrm{Hom}_{\mathfrak{P}_0}(\rho_0, \tau)$. The functor $m_L$ is an equivalence of categories. Let $f \in m_L(\tau), \gamma \in \mathcal{H}(L,\rho_0)$ and $w \in W$. The action of $\mathcal{H}(L,\rho_0)$ on $m_L(\tau)$ is given by $(\gamma. f)(w)= \int_L \tau(l)f(\gamma^{\vee}(l^{-1})w)dl$. Here $\gamma^{\vee}$ is defined on $L$ by $\gamma^{\vee}(l^{-1})=\gamma(l)^{\vee}$ for $l \in L$. Let $\tau' \in \mathfrak{R}^{[L,\pi]_G}(G)$ then the functor $m_G \colon \mathfrak{R}^{[L,\pi]_G}(G) \longrightarrow \mathcal{H}(G,\rho)-Mod$ is given by $m_G(\tau')=\mathrm{Hom}_{\mathfrak{P}}(\rho, \tau')$. The functor  $m_G$ is an equivalence of categories. From Corollary 8.4 in \cite{MR1643417}, the functors $m_L, m_G, Ind_P^G, (t_P)_*$  fit into the following commutative diagram:

\[
\begin{CD}
\mathfrak{R}^{[L,\pi]_G}(G)    @>m_G>>    \mathcal{H}(G,\rho)-Mod\\
@AInd_P^GAA                                    @A(t_P)_*AA\\
\mathfrak{R}^{[L,\pi]_L}(L)    @>m_L>>     \mathcal{H}(L,\rho_0)-Mod
\end{CD}
\]\par

If $\tau \in \mathfrak{R}^{[L,\pi]_L}(L)$ then from the above commutative diagram, we see that $(t_P)_*(m_L(\tau)) \cong m_G(Ind_P^G \tau)$ as $\mathcal{H}(G, \rho)$-modules. Replacing $\tau$ by $(\tau \otimes \delta_P^{1/2})$ in the above expression, $(t_P)_*(m_L(\tau \otimes \delta_P^{1/2})) \cong m_G(Ind_P^G (\tau \otimes \delta_P^{1/2}))$ as $\mathcal{H}(G, \rho)$-modules. As $Ind_P^G(\tau \otimes \delta_P^{1/2})= \iota_P^G(\tau)$, we have $(t_P)_*(m_L(\tau \otimes \delta_P^{1/2})) \cong m_G(\iota_P^G(\tau))$ as $\mathcal{H}(G, \rho)$-modules.\par

Our aim is to find an algebra embedding $T_P \colon \mathcal{H}(L,\rho_0) \longrightarrow \mathcal{H}(G, \rho)$ such that the following diagram commutes:

\[
\begin{CD}
\mathfrak{R}^{[L,\pi]_G}(G)    @>m_G>>    \mathcal{H}(G,\rho)-Mod\\
@A\iota_P^GAA                                    @A(T_P)_*AA\\
\mathfrak{R}^{[L,\pi]_L}(L)    @>m_L>>     \mathcal{H}(L,\rho_0)-Mod
\end{CD}
\]\par

Let $\tau \in \mathfrak{R}^{[L,\pi]_L}(L)$ then $m_L(\tau) \in \mathcal{H}(L,\rho_0)$- Mod. The functor $(T_P)_*$ is defined as below:

\[
(T_P)_*(m_L(\tau))= \left\lbrace \psi \colon \mathcal{H}(G,\rho) \to m_L(\tau) \; \middle|  \;
\begin{varwidth}{\linewidth}
$h\psi(h_1)= \psi(T_P(h)h_1)$ where\\
$h \in \mathcal{H}(L,\rho_0),h_1 \in \mathcal{H}(G,\rho)$
\end{varwidth}
\right \rbrace.
\]\par

From the above commutative diagram, we see that $(T_P)_*(m_L(\tau)) \cong m_G(\iota_P^G(\tau))$ as $\mathcal{H}(G, \rho)$-modules. Recall that $(t_P)_*(m_L(\tau \otimes \delta_P^{1/2})) \cong m_G(\iota_P^G(\tau))$ as $\mathcal{H}(G, \rho)$-modules. Hence we have to find an algebra embedding  $T_P \colon \mathcal{H}(L,\rho_0) \longrightarrow \mathcal{H}(G, \rho)$ such that 
$(T_P)_*(m_L(\tau)) \cong (t_P)_*(m_L(\tau \otimes \delta_P^{1/2}))$ as $\mathcal{H}(G, \rho)$-modules.\par

\begin{proposition}\label{pro_100}
	The map $T_P$ is given by $T_P(\phi)= t_P(\phi \delta_P^{-1/2})$ for $\phi \in \mathcal{H}(L, \rho_0)$ so that we have $(T_P)_*(m_L(\tau))= (t_P)_*(m_L(\tau \otimes \delta_P^{1/2}))$ as $\mathcal{H}(G,\rho)$- modules.
\end{proposition}
\begin{proof}
	Let $W$ be space of $\rho_0$. The vector spaces for $m_L(\tau \delta_P^{1/2})$ and $m_L(\tau)$ are the same. Let $f \in m_L(\tau)= \mathrm{Hom}_{\mathfrak{P}_0}(\rho_0, \tau),\gamma \in \mathcal{H}(L,\rho_0)$ and $w \in W$. Recall the action of $\mathcal{H}(L,\rho_0)$ on $m_L(\tau)$ is given by
	\[(\gamma. f)(w)= \int_L \tau(l)f(\gamma^{\vee}(l^{-1})w)dl.\] \par
	
	Let $f' \in m_L(\tau \delta_P^{1/2})= \mathrm{Hom}_{\mathfrak{P}_0}(\rho_0, \tau\delta_P^{1/2}),\gamma \in \mathcal{H}(L,\rho_0)$ and $w \in W$. Recall the action of $\mathcal{H}(L,\rho_0)$ on $m_L(\tau \delta_P^{1/2})$ is given by  \[(\gamma. f')(w)= \int_L  (\tau\delta_P^{1/2})(l)f'(\gamma^{\vee}(l^{-1})w)dl= \int_L \tau(l)\delta_P^{1/2}(l)f'(\gamma^{\vee}(l^{-1})w)dl.\] Now $f'$ is a linear transformation from space of $\rho_0$ to space of $\tau\delta_P^{1/2}$. As $\delta_P^{1/2}(l) \in \mathbb{C}^{\times}$, so $\delta_P^{1/2}(l)f'(\gamma^{\vee}(l^{-1})w)=f'(\delta_P^{1/2}(l)\gamma^{\vee}(l^{-1})w)$. Hence we have
	
	\[(\gamma. f')(w)= \int_L \tau(l)f'(\delta_P^{1/2}(l)\gamma^{\vee}(l^{-1})w)dl=  \int_L \tau(l)f'(\delta_P^{1/2}(l)\gamma(l)^{\vee}w)dl.\] Further as $\delta_P^{1/2}(l) \in \mathbb{C}^{\times}$, so  $\delta_P^{1/2}(l)(\gamma(l))^{\vee}=(\delta_P^{1/2}\gamma)(l)^{\vee}$. Therefore
	\[(\gamma.f')(w)= \int_L \tau(l)f'((\delta_P^{1/2}\gamma)(l)^{\vee}w)dl= (\delta_P^{1/2}\gamma).f'(w).\] Hence we can conclude that the action of $\gamma \in \mathcal{H}(L, \rho_0)$ on $f' \in m_L(\tau\delta_P^{1/2})$ is same as the action of $\delta_P^{1/2}\gamma \in \mathcal{H}(L, \rho_0)$  on $f' \in m_L(\tau)$. So we have  $(T_P)_*(m_L(\tau))= (t_P)_*(m_L(\tau \otimes \delta_P^{1/2}))$ as $\mathcal{H}(G,\rho)$- modules.
\end{proof}\par

From Proposition ~\ref{pro_100}, $T_P(\alpha)=t_P(\alpha \delta_P^{-1/2})$. So we have,

\begin{align*}
T_P(\alpha)&=t_P(\alpha \delta_P^{-1/2})\\
&= t(\alpha \delta_P^{-1/2}\delta_P)\\
&= t(\alpha \delta_P^{1/2})\\
&= T^+(\alpha \delta_P^{1/2}).
\end{align*}\par

Hence
\begin{align*}
T_P(\alpha)(\zeta)&= T^+(\alpha \delta_P^{1/2})(\zeta)\\
&=\delta_P^{1/2}(\zeta)T^+(\alpha)(\zeta)\\
&=\delta_P^{1/2}(\zeta)\alpha(\zeta)\\
&=\delta_P^{1/2}(\zeta)1_{W^{\vee}}.
\end{align*}

Thus $T_P(\alpha)(\zeta)= \delta_P^{1/2}(\zeta)1_{W^{\vee}}$ with $\text{supp}(T_P(\alpha))=\text{supp}(t_P(\alpha))=\mathfrak{P}\zeta\mathfrak{P}$.
\par

\subsection{Depth zero supercuspidal representations}

Suppose $\tau$ is an irreducible cuspidal representation of $\mathrm{GL}_n(k_E)$ inflated to a representation of $\mathrm{GL}_n(\mathfrak{O}_E)= K_0$. Then let $\widetilde{K_0}= ZK_0$ where $Z=Z(\mathrm{GL}_n(E))=\{\lambda \,1_n\mid \lambda \in E^{\times}\}$. As any element of $E^{\times}$ can be written as $u\varpi_E^n$ for some $u \in \mathfrak{O}_E^{\times}$ and $m \in \mathbb{Z}$. So in fact, $\widetilde{K_0}= <\varpi_E 1_n>K_0$. \par

Let $(\lambda, V)$ be a smooth irreducible supercuspidal representation of $\mathrm{GL}_n(E)$ such that $\lambda|_{K_0}=\tau$. Set $1_V$ to be the identity linear transformation of $V$. As $\varpi_E 1_n \in Z$, so $\lambda(\varpi_E 1_n)= \omega_{\lambda}(\varpi_E 1_n) 1_V$ where $\omega_{\lambda}\colon Z \longrightarrow \mathbb{C}^{\times}$ is the central character of $\lambda$.\par

Let $\widetilde\tau$ be a representation of $\widetilde K_0$ such that:
\begin{enumerate} 
	\item $\widetilde\tau(\varpi_E 1_n)=\omega_{\lambda}(\varpi_E 1_n) 1_V,$
	\item $\widetilde\tau|_{K_0}= \tau.$
\end{enumerate}
\par

Say $\omega_{\lambda}(\varpi_E 1_n)= z$ where $z \in \mathbb{C^{\times}}$. Now call $\widetilde\tau = \widetilde\tau_z$. We have extended $\tau$ to $\widetilde\tau_z$ which is a representation of $\widetilde K_0$, so that $Z$ acts by $\omega_\lambda$. Hence $\lambda|_{\widetilde K_0} \supseteq\widetilde \tau_z$ which implies that $\mathrm{Hom}_{\widetilde K_0} (\widetilde \tau_z, \lambda|_{\widetilde K_0}) \neq 0$.\par

By Frobenius reciprocity for induction from open subgroups, 

\begin{center}
	$\mathrm{Hom}_{\widetilde K_0} (\widetilde \tau_z, \lambda|_{\widetilde K_0}) \simeq \mathrm{Hom}_{\mathrm{GL}_n(E)}(c$-$Ind_{\widetilde K_0}^{\mathrm{GL}_n(E)} \widetilde\tau_z, \lambda)$.
\end{center}

Thus $\mathrm{Hom}_{\mathrm{GL}_n(E)}(c$-$Ind_{\widetilde K_0}^{\mathrm{GL}_n(E)} \widetilde\tau_z, \lambda) \neq 0$. So there exists a non-zero $\mathrm{GL}_n(E)$-map from $c$-$Ind_{\widetilde K_0}^G \widetilde\tau_z$ to $\lambda$. As $\tau$ is cuspidal representation, using Cartan decompostion and Mackey's criteria we can show that $c$-$Ind_{\widetilde K_0}^{\mathrm{GL}_n(E)} \widetilde\tau_z$ is irreducible. So $\lambda \simeq c$-$Ind_{\widetilde K_0}^{\mathrm{GL}_n(E)} \widetilde\tau_z$. As $c$-$Ind_{\widetilde K_0}^{\mathrm{GL}_n(E)} \widetilde\tau_z$ is irreducible supercuspidal representation of $\mathrm{GL}_n(E)$ of depth zero, so $\lambda$ is irreducible supercuspidal representation of $\mathrm{GL}_n(E)$ of depth zero.\par

Conversely, let $\lambda$ is a irreducible, supercuspidal, depth zero representation of  $\mathrm{GL}_n(E)$. So $\lambda^{K_1} \neq \{0\}$. Hence $\lambda|_{K_1} \supseteq 1_{K_1}$, where $1_{K_1}$ is trivial representation of $K_1$. This means $\lambda|_{K_0} \supseteq \tau$, where $\tau$ is an irreducible representation of $K_0$ such that $\tau|_{K_1} \supseteq 1_{K_1}$. So $\tau$ is trivial on $K_1$. So $\lambda|_{K_0}$ contains an irreducible representation $\tau$ of $K_0$ such that $\tau|_{K_1}$ is trivial. So $\tau$ can be viewed as an irreducible representation of $K_0/K_1\cong \mathrm{GL}_n(k_E)$  inflated to $K_0=\mathrm{GL}_n(\mathfrak{O}_E)$. The representation $\tau$ is cuspidal by (a very special case of) A.1 Appendix \cite{MR1235019}.\par

So we have the following bijection of sets:

\[
\left\lbrace 
\begin{varwidth}{\linewidth}
Isomorphism classes of irreducible \\ cuspidal 
representations of $\mathrm{GL}_n(k_E)$
\end{varwidth}
\right \rbrace		
\times \mathbb{C^{\times}} \longleftrightarrow 
\left\lbrace 
\begin{varwidth}{\linewidth}
Isomorphism classes \\of irreducible  \\
supercuspidal \\representations of \\ 
$\mathrm{GL}_n(E)$ of depth zero

\end{varwidth}
\right \rbrace.
\]
\par

\[(\tau, z)  \xrightarrow{\hspace*{6cm}}  c-Ind_{\widetilde K_0}^{\mathrm{GL}_n(E)} \widetilde\tau_z \]

\[(\tau, \omega_\lambda(\varpi_E 1_n)) \xleftarrow{\hspace*{6cm}}  \lambda \]\par

Recall that $\pi$ is an irreducible supercuspidal depth zero representation of $L \cong \mathrm{GL}_n(E)\times \mathrm{U}_1(E)$. So $\pi= \lambda \chi$ where $\lambda$ is an irreducible supercuspidal depth zero representation of $\mathrm{GL}_n(E)$ and $\chi$ is an irreducible supercuspidal depth zero character of $\mathrm{U}_1(E)$. From now on we denote the representation $\tau \chi$ by $\rho_0$. So $\rho_0$ is an irreducible cuspidal representation of $\mathrm{GL}_n(k_E) \times \mathrm{U}_1(k_E)$ inflated to $K_0 \times \mathrm{U}_1(\mathfrak{O}_E)$ where $K_0=\mathrm{GL}_n(\mathfrak{O}_E)$. Recall that we can extend $\rho_0$ to a representation $\widetilde{\rho_0}$ of $Z(L)\mathfrak{P}_0= \coprod_{n \in \mathbb{Z}} \mathfrak{P}_0 \zeta^n$ via $\widetilde{\rho_0}(\zeta^k j)=\rho_0(j)$ for $j \in \mathfrak{P}_0, k \in \mathbb{Z}$. Also observe that as $\lambda= c$-$Ind_{\widetilde K_0}^{\mathrm{GL}_n(E)}\widetilde{\tau}$, so $\pi= \lambda \chi \simeq c$-$Ind_{Z(L)\mathfrak{P}_0}^L \widetilde{\rho_0}$.

\section{Representation $\rho$ of $\mathfrak{P}$} \label{sec_3}

 Let $V$ be the vector space associated with $\rho_0$. Now $\rho_0$ is extended to a map $\rho$ from $\mathfrak{P}$ to $GL(V)$ as follows. By Iwahori factorization, if $j \in \mathfrak{P}$ then $j$ can be written as $j_{-}j_{0}j_{+}$, where $j_{-}\in \mathfrak{P}_{-}, j_{+}\in \mathfrak{P}_{+}, j_{0}\in \mathfrak{P}_{0}$. Now the map $\rho$ on $\mathfrak{P}$ is defined as $\rho(j)= \rho_0(j_0)$.\par 
					
\begin{proposition}\label{pro_11}
$\rho$ is a homomorphism from $\mathfrak{P}$ to $GL(V)$. So $\rho$  becomes a representation of $\mathfrak{P}$.
\end{proposition}
					
\begin{proof}
The proof goes in similar lines as Proposition 5 in \cite{sandeep}.
\end{proof}

\section{Calculation of $N_G(\mathfrak{P}_0)$} \label{sec_4}

We set $G=\mathrm{U}(n,n+1)$. To describe $\mathcal{H}(G,\rho)$ we need to determine $N_G(\rho_0)$ which is given by
\[ N_G(\rho_0)=\{ m \in N_G(\mathfrak{P}_0) \mid {\rho_0}\simeq \rho_0^m \}.\]  

Further, to find out $N_G(\rho_0)$ we need to determine $N_G(\mathfrak{P}_0)$. To that end we shall calculate $N_{\mathrm{GL}_n(E)}(K_0)$. Let $Z= Z(\mathrm{GL}_n(E))$. So $Z=\{\lambda 1_n \mid \lambda \in E^{\times} \}$.

\begin{lemma}\label{Normalizer_of_K_0_in_GL_n(E)} 
	$N_{\mathrm{GL}_n(E)}(K_0)= K_0Z$. 
\end{lemma}
\begin{proof}
	This follows from the Cartan decomposition by a direct matrix calculation.
\end{proof}\par

Now let us calculate $N_G(\mathfrak{P}_0)$. Note that 
$J=
\begin{bmatrix}
	0 &0 &Id_{n}\\
	0& 1 & 0 \\
	Id_{n} & 0 & 0
	\end {bmatrix} \in G$. Indeed, $J \in N_G(\mathfrak{P}_0)$. The center $Z(\mathfrak{P}_0)$ of $\mathfrak{P}_0$ is given by
\[
Z(\mathfrak{P}_0)=\Bigg\lbrace
\begin{bmatrix}
uId_n  & 0 &0\\
0 & \lambda &0\\
0    & 0 & \overline{u}^{-1}Id_n \\
\end{bmatrix} \mid u \in \mathfrak{O}_E^{\times}, \lambda \in \mathfrak{O}_E^{\times}, \lambda\overline{\lambda}=1 \Bigg \rbrace.
\]

Recall the center $Z(L)$ of $L$ is given by
\[
Z(L)=\Bigg\lbrace
\begin{bmatrix}
aId_n  & 0 &0\\
0 & \lambda &0\\
0    & 0 & \overline{a}^{-1}Id_n \\
\end{bmatrix} \mid a \in E^{\times}, \lambda \in E^{\times},  \lambda\overline{\lambda}=1 \Bigg \rbrace.
\]

\begin{proposition}\label{pro_12}
	$N_G(\mathfrak{P}_0)=\left<\mathfrak{P}_0 Z(L), J \right>= \mathfrak{P}_0Z(L) \rtimes \left<J \right>$.
\end{proposition}

\begin{proof}
We use Lemma \ref{Normalizer_of_K_0_in_GL_n(E)} to prove this Proposition. The proof goes in the similar lines as Proposition 6 in \cite{sandeep}.
\end{proof}

\section{Calculation of $N_G(\rho_0)$} \label{Sec_5}

\subsection{Unramified case}
We have the following conclusion about $N_G(\rho_0)$ for the unramified case:

If $n$ is even then $N_G(\rho_0)= Z(L) \mathfrak{P}_0$ and if $n$ is odd then $N_G(\rho_0)= Z(L) \mathfrak{P}_0 \rtimes \langle J \rangle$. For details refer to section 5.1 in \cite{sandeep}. 

\subsection{Ramified case:}

We have the following conclusion about $N_G(\rho_0)$ for ramified case:

If $n$ is odd then $N_G(\rho_0)= Z(L) \mathfrak{P}_0$ and if $n$ is even then $N_G(\rho_0)= Z(L) \mathfrak{P}_0 \rtimes \langle J \rangle$. For details refer to section 5.2 in \cite{sandeep}. 	

\begin{lemma}\label{lem_2}
	When $n$ is odd in the unramified case or when $n$ is even in the ramified case, we have $N_G(\rho_0)= \left<\mathfrak{P}_0, w_0, w_1\right>$, where $w_0=J$ and $w_1=
	\begin{bmatrix}
	0 & 0& {\overline{\varpi}_E}^{-1}Id_n\\
	0&1&0\\
	\varpi_E Id_n & 0 &0
	\end{bmatrix}$.\\
\end{lemma}
\begin{proof}
The proof goes in the similar lines as Lemma 2 in \cite{sandeep}.
\end{proof}

\section{Structure of $\mathcal{H}(G,\rho)$} \label{sec_6} 

\subsection{Unramified case:}

In this section, we determine the structure of $\mathcal{H}(G,\rho)$ for the unramified case when $n$ is odd. Using cuspidality of $\rho_0$, it can be shown by Theorem 4.15 in \cite{MR1235019}, that $ \mathfrak{I}_G(\rho)= \mathfrak{P}N_G(\rho_0)\mathfrak{P}$. But from lemma \ref{lem_2}, $N_G(\rho_0)= \left<\mathfrak{P}_0, w_0, w_1\right>$. So $\mathfrak{I}_G(\rho)=\mathfrak{P}\left<\mathfrak{P}_0, w_0, w_1\right>\mathfrak{P}= \mathfrak{P}\left<w_0, w_1\right> \mathfrak{P}$, as $\mathfrak{P}_0$ is a subgroup of $\mathfrak{P}$. Let $V$ be the vector space corresponding to $\rho$. Let us recall that  $\mathcal{H}(G, \rho)$ consists of maps $f \colon G \to End_{\mathbb{C}}(V^{\vee})$ such that support of $f$ is compact and $f(pgp')= \rho^{\vee}(p)f(g)\rho^{\vee}(p')$ for $p,p' \in \mathfrak{P}, g \in G$. In fact $\mathcal{H}(G,\rho)$ consists of $\mathbb{C}$-linear combinations of maps $f \colon G \longrightarrow End_{\mathbb{C}}(V^{\vee})$ such that $f$ is supported on $\mathfrak{P}x\mathfrak{P}$ where $x \in \mathfrak{I}_G(\rho)$ and $f(pxp')= \rho^{\vee}(p)f(x)\rho^{\vee}(p')$ for $p,p' \in \mathfrak{P}$. We shall now show there exists $\phi_0 \in \mathcal{H}(G,\rho)$ with support $\mathfrak{P}w_0\mathfrak{P}$ and satisfies $\phi_0^2= q^n+ (q^n-1)\phi_0$. Let

\[
K(0)= \mathrm{U}(n,n+1) \cap \mathrm{GL}_{2n+1}(\mathfrak{O}_E)=\{g \in \mathrm{GL}_{2n+1}(\mathfrak{O}_E)\mid {}^t\overline{g}Jg=J\},
\]
\[
K_1(0)=\{g \in Id_{n+1}+ \varpi_E \mathrm{M}_{2n+1}(\mathfrak{O}_E)\mid {}^t\overline{g}Jg=J\},
\]
\[
\mathsf{G}= \{g \in \mathrm{GL}_{2n+1}(k_E) \mid {}^t\overline{g}Jg=J\}.
\]

The map $r$ from $K(0)$ to $\mathsf{G}$ given by $r\colon K(0)\xrightarrow{\text{mod} \, p_E}\mathsf{G}$ is a surjective group homomorphism with kernel $K_1(0)$. So by the first isomorphism theorem of groups we have:

\begin{center}
	$\frac{K(0)}{K_1(0)}\cong \mathsf{G}.$
\end{center}

$r(\mathfrak{P})= \mathsf{P}=
\begin{bmatrix}
\mathrm{GL}_n(k_E) &M_{n \times 1}(k_E) &\mathrm{M}_n(k_E) \\
0 & \mathrm{U}_1(k_E) & M_{1 \times n}(k_E) \\
0         & 0& \mathrm{GL}_n(k_E)
\end{bmatrix} \bigcap \mathsf{G}$= Siegel parabolic subgroup of $\mathsf{G}$.\\\par

Now $\mathsf{P}= \mathsf{L} \ltimes \mathsf{U}$, where $\mathsf{L}$ is the Siegel Levi component of $\mathsf{P}$ and $\mathsf{U}$ is the unipotent radical of $\mathsf{P}$. Here

\[
\mathsf{L}= \Bigg\lbrace \begin{bmatrix}
a & 0 & 0\\
0 &\lambda &0\\
0 &0 &{^t}{\overline {a}}{^{-1}}
\end{bmatrix} \mid a \in \mathrm{GL}_n(k_E), \lambda \in k_E^{\times}, \lambda\overline{\lambda}=1 \Bigg \rbrace,
\]

\[
\mathsf{U}= \Bigg \lbrace \begin{bmatrix}
Id_n &u &X \\
0 &1 &-{}^t\overline{u}\\
0 &0 &Id_n
\end{bmatrix} \mid  X \in \mathrm{M}_n(k_E), u \in \mathrm{M}_{n \times 1}(k_E), X + {}^t\overline{X}+u{}^t\overline{u}=0 \Bigg \rbrace.
\]

Let $V$ be the vector space corresponding to $\rho$. The Hecke algebra $\mathcal{H}(K(0),\rho)$ is a sub-algebra of $\mathcal{H}(G,\rho)$.\par

Let $\overline{\rho}$ be the representation of $\mathsf{P}$ which when inflated to $\mathfrak{P}$ is given by $\rho$ and $V$ is also the vector space corresponding to $\overline{\rho}$. The Hecke algebra $\mathcal{H}(\mathsf{G},\overline{\rho})$ looks as follows:

\[
\mathcal{H}(\mathsf{G},\overline{\rho})= \left\lbrace f \colon \mathsf{G} \to End_{\mathbb{C}}(V^{\vee}) \; \middle|  \;
\begin{varwidth}{\linewidth}
$f(pgp')= \overline{\rho}^{\vee}(p)f(g)\overline{\rho}^{\vee}(p')$\\
where $p,p' \in \mathsf{P}, \, g \in \mathsf{G}$
\end{varwidth}
\right \rbrace.
\] \par

Now the homomorphism $r \colon K(0) \longrightarrow \mathsf{G}$ extends to a map from  $\mathcal{H}(K(0),\rho)$ to $\mathcal{H}(\mathsf{G},\overline{\rho})$ which we again denote by $r$. Thus $r \colon \mathcal{H}(K(0),\rho) \longrightarrow \mathcal{H}(\mathsf{G},\overline{\rho})$ is given by

\[
r(\phi)(r(x))=\phi(x)
\]
\[ 
\text{for} \, \phi \in \mathcal{H}(K(0),\rho) \, \text{and} \, x \in K(0).
\]
\begin{proposition}\label{pro_17}
	The map $r \colon \mathcal{H}(K(0),\rho) \longrightarrow \mathcal{H}(\mathsf{G},\overline{\rho})$ is an algebra isomorphism.
\end{proposition}
\begin{proof}
 Refer to Proposition 17 in \cite{sandeep}
\end{proof}

Let $w=r(w_0)=r(
\begin{bmatrix}
0 &0& Id_n \\
0&1&0\\
Id_n & 0&0
\end{bmatrix})=
\begin{bmatrix}
0 &0& Id_n \\
0&1&0\\
Id_n & 0&0
\end{bmatrix} \in \mathsf{G}$. Observe that $K(0) \supseteq \mathfrak{P}\amalg\mathfrak{P}w_0\mathfrak{P}$ and $\mathsf{G} \supseteq \mathsf{P}\amalg\mathsf{P}w_0\mathsf{P}$.\par 

The induced representation $Ind_{\mathsf{P}}^{\mathsf{G}}\overline{\rho}$ is a sum of two irreducible subrepresentations (by general theory). The ratio of the dimensions of these subrepresentations gives a parameter in the Hecke algebra. This is part of Howlett-Lehrer's general theory. Kutzko-Morris reworks this key observation. Hence we have $Ind_{\mathsf{P}}^{\mathsf{G}}\overline{\rho}= \pi_1 \oplus \pi_2$, where $\pi_1, \pi_2$ are distinct irreducible representations of $\mathsf{G}$ with $\text{dim} \pi_2 \geqslant \text{dim} \pi_1$. Let $\lambda= \frac{\text{dim} \pi_2}{\text{dim} \pi_1}$. By Proposition 3.2 in \cite{MR2276353}, there exists a unique $\phi$ in $\mathcal{H}(\mathsf{G},\overline{\rho})$ with support $\mathsf{P}w\mathsf{P}$ such that $\phi^2= \lambda + (\lambda-1)\phi$. By Proposition \ref{pro_17}, there is a unique element $\phi_0$ in $\mathcal{H}(K(0),\rho)$ such that $r(\phi_0)=\phi$. Thus supp($\phi_0$)=$\mathfrak{P}w_0\mathfrak{P}$ and $\phi_0^2= \lambda+(\lambda-1)\phi_0$. As support of $\phi_0= \mathfrak{P}w_0\mathfrak{P}\subseteq K(0)\subseteq G$, so $\phi_0$ can be extended to $G$ and viewed as an element of $\mathcal{H}(G,\rho)$. Thus $\phi_0$ satisfies the following relation in $\mathcal{H}(G,\rho)$:

\[
\phi_0^2= \lambda+ (\lambda-1)\phi_0.
\]
\par  

We shall now show that $\lambda= q^n$. Recall that as $\rho_0$ is an irreducible cuspidal representation of $\mathrm{GL}_n(k_E) \times \mathrm{U}_1(k_E)$, so $\rho_0= \tau_{\theta} \chi$, where $\tau_{\theta}$ is an irreducible cuspidal representation of $\mathrm{GL}_n(k_E)$ and $\chi$ is a cuspidal representation of $\mathrm{U}_1(k_E)$. Note that here $\theta$ is a regular charcter of $l^{\times}$ where $[ l \colon k_E ]=n$ and $k_E= \mathbb{F}_{q^2}$ so that $l= \mathbb{F}_{q^{2n}}$. Recall that $\theta^{\Phi}= \theta^{q^2}$. Hence, from Proposition 8 in \cite{sandeep} we have, $\theta^{q^n}=\theta^{-1}$ .\par

As $\mathsf{G}= \mathrm{U}(n,n+1)(k_E)$, so the dual group $\mathsf{G}^{*}$ is given by $\mathsf{G}^{*} \cong \mathrm{U}(n,n+1)(k_E)$ (i.e $\mathsf{G}^{*} \cong \mathsf{G}$). Note that $\theta$ corresponds to a semi-simple element $s^{*} \in L^{*}$ in $\mathsf{G}^{*}$. Then by Theorems 8.4.8 and 8.4.9 in \cite{carter_1992}, we have $\lambda = |c_{\mathsf{G}^{*}}(s^{*})|_p$.\par

Note that $L^{*} \cong L$. So $s^{*}$ corresponds to $s$ in $L$. Hence, we have $\lambda = |c_{\mathsf{G}}(s)|_p$. We write $s = \begin{bmatrix} 
\alpha &0 &0\\
0 &\lambda &0\\
0 &0 & {}^t\overline{\alpha}^{-1} 
\end{bmatrix}$. Observe that $\lambda \overline{\lambda}=1, \lambda \in k_E^{\times}, \alpha \in \mathbb{F}_{q^{2n}}^{\times}$. More precisely, $\alpha$ is in the image of 
$\mathbb{F}_{q^{2n}}^{\times}$ under a fixed embedding $\mathbb{F}_{q^{2n}}^{\times} \hookrightarrow \mathrm{GL}_n(\mathbb{F}_{q^2})$. This embedding arises when we let $l$ act on the basis of $l$ over $k_E$ via multiplication. We can thus embed $l$ in $M_n(k_E)$ and $l^{\times}$ in $\mathrm{GL}_n(k_E)$ which we call the usual embedding. Note that $\theta$ is regular implies that $\mathbb{F}_{q^{2n}}= \mathbb{F}_{q^2}(\alpha)$. Our goal is to compute $|c_{\mathsf{G}}(s)|_p$.\par  

By Proposition 3.19 in \cite{digne_francois_michel_1991}, we have Sylow p-subgroups of $c_{\mathsf{G}}(s)$ are the sets of $\mathbb{F}_{q^2}$-points of the Unipotent radicals of the Borel subgroups of $c_{\mathsf{G}}(s)$. By Proposition 2.2 in \cite{digne_francois_michel_1991}, we have Borel subgroups of $c_{\mathsf{G}}(s)$ are of the form $B \cap c_{\mathsf{G}}(s)$, where $B$ is a Borel subgroup of $\mathsf{G}$. As Siegel parabolic subgroup $\mathsf{P}$ of $\mathsf{G}$ contains a Borel subgroup of $\mathsf{G}$, so $c_{\mathsf{P}}(s)= \mathsf{P}
 \cap c_{\mathsf{G}}(s)$ contains a Sylow p-subgroup of $c_{\mathsf{G}}(s)$. \par
 
 \begin{lemma} \label{semidirect}
$ c_{\mathsf{P}}(s)=c_{\mathsf{L}}(s) \ltimes c_{\mathsf{U}}(s)$.
 \end{lemma}
 \begin{proof}
Recall that $\mathsf{P}=\mathsf{L} \ltimes \mathsf{U}$. Hence $\mathsf{L} \cap \mathsf{U}= \varnothing$ and $\mathsf{U} \unlhd \mathsf{P}$. As  $\mathsf{L} \cap \mathsf{U}= \varnothing \Longrightarrow c_{\mathsf{L}}(s) \cap c_{\mathsf{U}}(s)= \varnothing $. Note that $c_{\mathsf{U}}(s) \unlhd  (c_{\mathsf{L}}(s) \times c_{\mathsf{U}}(s))$. So it makes sense to talk of $c_{\mathsf{L}}(s) \ltimes c_{\mathsf{U}}(s)$. \par

Let $x \in \mathsf{P}(s) \Longrightarrow x \in \mathsf{P}, sxs^{-1}=x$. Note that as $x \in \mathsf{P}$ so $x=lu$ for some $l \in \mathsf{L}, u \in \mathsf{U}$. Therefore,

\begin{align*}
slus^{-1}=lu \\
\Longrightarrow sls^{-1}sus^{-1}=lu.\\
\end{align*}

Let $sls^{-1} =m$ and $sus^{-1}=n$. Now as $s \in \mathsf{L}$, so $sls^{-1}=m \in \mathsf{L}$. Note that $sus^{-1}=n \in \mathsf{U}$ as $\mathsf{U} \unlhd \mathsf{P}$. Therefore, we have $mn=lu$ or $m^{-1}l = nu^{-1}$. But $m^{-1}l \in \mathsf{L}$ and $nu^{-1} \in \mathsf{U}$, so we have $m^{-1}l, nu^{-1} \in \mathsf{L} \cap \mathsf{U}$. Recall that $\mathsf{L} \cap \mathsf{U} = {e}$, so $m=l, n=u$. Therefore, $sls^{-1}=l, sus^{-1}=u$. So we have $\l \in c_{\mathsf{L}}(s), u \in c_{\mathsf{U}}(s)$. Hence, $x \in c_{\mathsf{L}}(s) \ltimes c_{\mathsf{U}}(s)$. So $c_{\mathsf{P}}(s) \subseteq c_{\mathsf{L}}(s) \ltimes c_{\mathsf{U}}(s)$. \par

Conversely, let $x \in c_{\mathsf{L}}(s) \ltimes c_{\mathsf{U}}(s)$. So $x =lu$ where $ l\in c_{\mathsf{L}}(s)$ and $u \in c_{\mathsf{U}}(s)$. Hence $sls^{-1}=l$ and $sus^{-1}=u$. Therefore, $sxs^{-1}=slus^{-1}=sls^{-1}sus^{-1}=lu=x$. So $x \in c_{\mathsf{P}}(s)$. Hence $c_{\mathsf{L}}(s) \ltimes c_{\mathsf{U}}(s) \subseteq c_{\mathsf{P}}(s)$. Therefore, $c_{\mathsf{P}}(s)=c_{\mathsf{L}}(s) \ltimes c_{\mathsf{U}}(s)$.
\end{proof}

From lemma \ref{semidirect}, we get $|c_{\mathsf{P}}(s)|_p= |c_{\mathsf{L}}(s)|_p |c_{\mathsf{U}}(s)|_p$. Note that $|c_{\mathsf{L}}(s)|_p=1$. Therefore, $|c_{\mathsf{P}}(s)|_p= |c_{\mathsf{L}}(s)|_p |c_{\mathsf{U}}(s)|_p= |c_{\mathsf{U}}(s)|_p$. \par

 \begin{lemma} \label{cardinality}
$|c_{\mathsf{U}}(s)|= |c_{\mathsf{U}}(s)|_p= q^n$.	
\end{lemma}

\begin{proof}
Recall that the elements of $\mathsf{U}$ are of form 

\[
m= \begin{bmatrix}
Id_n & u &X\\
0 &1& -{}^t \overline{u}\\
0 &0 &Id_n\\
\end{bmatrix}
\] where $x \in M_n(k_E), u \in M_{n\times 1}(k_E), X+{}^t\overline{X}+u{}^t\overline{u}=0$. If $m \in c_{\mathsf{U}}(s)$ then $ms=sm$. So we have,
\[
 \begin{bmatrix}
 \alpha & 0 &0\\
0 &\lambda& 0\\
0 &0 &{}^t\overline{\alpha}^{-1}
\end{bmatrix}
 \begin{bmatrix}
Id_n & u &X\\
0 &1& -{}^t \overline{u}\\
0 &0 &Id_n
\end{bmatrix}=
 \begin{bmatrix}
Id_n & u &X\\
0 &1& -{}^t \overline{u}\\
0 &0 &Id_n
\end{bmatrix}
 \begin{bmatrix}
\alpha & 0 &0\\
0 &\lambda& 0\\
0 &0 &{}^t\overline{\alpha}^{-1}\\
\end{bmatrix}.
\]

From the above matrix relation, it follows that $\alpha u= \lambda u, \alpha X= X {}^t \overline{\alpha}^{-1}, \lambda {}^t \overline{u}= {}^t \overline{u} {}^t \overline{\alpha}^{-1}$. Recall that $X+{}^t\overline{X}+u{}^t\overline{u}=0, \lambda \overline{\lambda}=1$. Also recall that $u \in M_{n \times 1}(k_E), \alpha \in \mathbb{F}_{q^{2n}}^{\times}, k_E(\alpha)=l$. As $\alpha u= \lambda u$, so if $u \neq 0$ then $\lambda \in k_E$ is an eigen value of $\alpha$. So $\lambda$ is a root of the minimal polynomial of $\alpha$ over $k_E$. But as the minimal polynomial is irreducible over $k_E[x]$, so this is a contradiction. So $u=0$.\par

So we have to find $X$ such that $X+{}^t\overline{X}=0, \alpha X= X {}^t \overline{\alpha}^{-1}$. Let  $\Xi = M_n(k_E)$ and set $\Xi_{\epsilon}= \{X \in \Xi \mid {}^t \overline{X}= \epsilon X \}$. Note that $X \in \Xi$ can be written as $\frac{X+{}^t \overline{X}}{2} + \frac{X-{}^t \overline{X}}{2}$, so $\Xi= \Xi_{1} \oplus \Xi_{-1}$.\par

Let us set $\Xi(\alpha)= \{X \in \Xi \mid \alpha X {}^t \overline{\alpha}=X \}$ and 
$\Xi_{\epsilon}(\alpha)= \{X \in \Xi_{\epsilon} \mid \alpha X {}^t \overline{\alpha}=X \}$. Then we have, $\Xi(\alpha)= \Xi_{1}(\alpha) \oplus \Xi_{-1}(\alpha)$. Let us choose $\gamma \in k_E$ such that $\gamma \neq 0$ and $\overline{\gamma}= -\gamma$. Note that, if $X \in \Xi_1(\alpha)$ then $X={}^tX$ and $\alpha X {}^t \overline{\alpha}=X$. So ${}^t(\overline{\gamma X})=-(\gamma X)$ and 
$\alpha (\gamma X) {}^t \overline{\alpha}= \gamma X$. Therefore, $\gamma X \in \Xi_{-1}(\alpha)$. We also have a bijection from $c_{\mathsf{U}}(s) \longrightarrow \Xi_1(\alpha)$ given by:

\[
 \begin{bmatrix}
Id_n & 0 &X\\
0 &1& 0\\
0 &0 &Id_n\\
\end{bmatrix} \longrightarrow X.
\]  

Hence we have, $|c_{\mathsf{U}}(s)|= |\Xi_1(\alpha)|=|\Xi_{-1}(\alpha)|$. Let us now compute $|\Xi(\alpha)|$. So we want to find the cardinality of $X \in \Xi$ such that $\alpha X {}^t \overline{\alpha}=X$ for a fixed $\alpha \in \mathbb{F}_{q^{2n}}^{\times}$. Let $\phi_1 \colon \mathbb{F}_{q^{2n}} \hookrightarrow M_n(\mathbb{F}_{q^2})$ be the usual embedding take $\beta$ to $m_{\beta}$. Let $f(x)$ be the minimal polynomial of $\alpha$ over $k_E= \mathbb{F}_{q^2}$. So we have $\mathbb{F}_{q^{2n}} \cong \frac {\mathbb{F}_{q^2}[x]}{<f(x)>}$. Hence, a polynomial $p(\alpha) \in k_E(\alpha)$  is mapped to $p(m_{\alpha})$. \par

Let us consider an another embedding $\phi_2 \colon  \mathbb{F}_{q^{2n}} \cong \frac {\mathbb{F}_{q^2}[x]}{<f(x)>} \hookrightarrow M_n(\mathbb{F}_{q^2})$ given by $\phi_2(\alpha)= {}^t \overline{m}_{\alpha}^{-1}$. We must show that $\phi_2$ is well-defined. That is, we have to show that $f({}^t \overline{m}_{\alpha}^{-1})=0$.
 But observe that, $f({}^t \overline{m}_{\alpha}^{-1})= 
 {}^t \overline{f(m_{\alpha}^{-1})}=  {}^t \overline{f(m_{\alpha}^{q^n})}={}^t \overline{(f(m_{\alpha}))^{q^n}}=0^{q^n}=0$. In the above relations, we have used the fact that $\theta^{-1}= \theta^{q^n}$ which follows from Proposition 8 in 
 \cite{sandeep}. Therefore, $\phi_2$ is well-defined.\par

 Hence we have two different embeddings $\phi_1$ and $\phi_2$ of $l$ in $M_n(q^2)$. Recall that, we want to compute the cardinality of $X \in \Xi$ such that $\alpha X {}^t \overline{\alpha}=X$ for a fixed $\alpha \in \mathbb{F}_{q^{2n}}^{\times}$. That is, we want to compute the cardinality of $X \in \Xi$ such that $X \phi_2(\lambda) =\phi_1(\lambda) X$ for $\lambda \in l=\mathbb{F}_{q^2n}$.\par
 
 Note that, we can make $V=k_E^n$ into a $l$-module in two different ways. Namely, for $\lambda \in l$ and $v \in V$ we have, 
 
\begin{align*}
 \lambda. v= \phi_1(\lambda).v\\
 \lambda* v= \phi_2(\lambda).v
\end{align*}

Let us denote the two $l$-modules by ${}_1k_E^n$ and ${}_2k_E^n$. So $X \phi_2(\lambda) =\phi_1(\lambda) X \Longleftrightarrow X \in \mathrm{Hom}_l({}_1k_E^n,{}_2k_E^n) \cong \mathrm{Hom}_l(l,l) \cong l$. Therefore, we have $|\Xi(\alpha)|= |\mathrm{Hom}_l({}_1k_E^n,{}_2k_E^n)|= |l|=q^{2n}$. \par

Note that $|\Xi(\alpha)|= |\Xi_{1}(\alpha)|.|\Xi_{-1}(\alpha)|$. But as $|\Xi_{1}(\alpha)|=|\Xi_{-1}(\alpha)|$, so we have $|\Xi(\alpha)|= |\Xi_{-1}(\alpha)|^2=q^{2n}$. Thus $|\Xi_{-1}(\alpha)|= q^n$. Therefore, $|c_{\mathsf{U}}(s)|_p=|c_{\mathsf{U}}(s)|=|\Xi_{-1}(\alpha)|=q^n$.\\ 
 
\end{proof}

From Lemmas \ref{cardinality} and \ref{semidirect} we have, $\lambda= |c_{\mathsf{U}}(s)|_p= |c_{\mathsf{L}}(s)|_p.|c_{\mathsf{L}}(s)|_p=1.q^n=q^n$.\par

Recall that $\phi_0 \in \mathcal{H}(G, \rho)$ has support $\mathfrak{P}w_0\mathfrak{P}$ and satisfies the relation $\phi_0^2= \lambda+(\lambda-1)\phi_0$. So we have $\phi_0^2= q^n+(q^n-1)\phi_0$ in $\mathcal{H}(G, \rho)$. 

Now we shall now show that there exists $\phi_1 \in \mathcal{H}(G, \rho)$ with support $\mathfrak{P}w_1\mathfrak{P}$ satisfying the same relation as $\phi_0$. Let $\eta \in \mathrm{U}(n,n+1)$ be such that $\eta w_0 \eta^{-1}= w_1$ and $\eta \mathfrak{P} \eta^{-1}= \mathfrak{P}$.\par

As $\mathfrak{P} \subseteq K(0)$ and $w_0 \in K(0)$, so $K(0) \supseteq \mathfrak{P} \amalg \mathfrak{P}w_0\mathfrak{P} \Longrightarrow \eta K(0) \eta^{-1} \supseteq \eta \mathfrak{P} \eta^{-1} \amalg \eta\mathfrak{P}w_0\mathfrak{P}\eta^{-1}$. But observe that  $\eta \mathfrak{P} \eta^{-1}= \mathfrak{P}$ and  $\eta\mathfrak{P}w_0\mathfrak{P}\eta^{-1}= (\eta\mathfrak{P}\eta^{-1})(\eta w_0 \eta^{-1})(\eta\mathfrak{P}\eta^{-1})=\mathfrak{P}w_1\mathfrak{P}$ (since $\eta w_0 \eta^{-1}=w_1$). So $\eta K(0) \eta^{-1} \supseteq \mathfrak{P} \amalg \mathfrak{P}w_1\mathfrak{P}$.\par

Let $r'$ be homomorphism of groups given by the map $r'\colon \eta K(0)\eta^{-1} \longrightarrow \mathsf{G}$ such that $r'(x)= (\eta^{-1}x\eta) mod p_E$ for $x \in \eta K(0) \eta^{-1}$. Observe that $r'$ is a surjective homomorphism of groups because $r'(\eta K(0) \eta^{-1})= (\eta^{-1} \eta K(0) \eta^{-1} \eta) mod p_E=  K(0) mod p_E =\mathsf{G}$. The kernel of group homomorphism is $\eta K_1(0) \eta^{-1}$. Now by the first isomorphism theorem of groups we have $\frac{\eta K(0) \eta^{-1}}{\eta K_1(0) \eta^{-1}} \cong \frac{K(0)}{K_1(0)} \cong \mathsf{G}$. Also $r'(\eta\mathfrak{P}\eta^{-1})=(\eta^{-1} \eta \mathfrak{P} \eta^{-1} \eta) mod p_E=\mathfrak{P}mod p_E= \mathsf{P}$. Let $\overline{\rho}$ be representation of $\mathsf{P}$ which when inflated to $\mathfrak{P}$ is given by $\rho$. The Hecke algebra of $\eta K(0) \eta^{-1}$ which we denote by $\mathcal{H}(\eta K(0) \eta^{-1}, \rho)$ is a sub-algebra of $\mathcal{H}(G,\rho)$.\par

The map $r':\eta K(0) \eta^{-1} \longrightarrow \mathsf{G}$ extends to a map from  $\mathcal{H}(\eta K(0) \eta^{-1},\rho)$ to $\mathcal{H}(\mathsf{G}, \overline{\rho})$ which we gain denote by $r'$. Thus $r' \colon \mathcal{H}(\eta K(0) \eta^{-1},\rho) \longrightarrow \mathcal{H}(\mathsf{G}, \overline{\rho})$ is given by
\[r'(\phi)(r'(x))= \phi(x) \]
\[\text{for} \, \phi \in \mathcal{H}(\eta K(0) \eta^{-1},\rho) \, \text{and} \, x \in \eta K(0) \eta^{-1}.\] \par

The proof that $r'$ is an isomorphism goes in the similar lines as Proposition \ref{pro_17}. We can observe that $r'(w_1)=w \in \mathsf{G}$, where $w$ is defined as before in this section. As we know from our previous discussion in this section, that there exists a unique $\phi$ in  $\mathcal{H}(\mathsf{G}, \overline{\rho})$ with support $\mathsf{P} w \mathsf{P}$ such that $\phi^2= q^n+ (q^n-1)\phi$. Hence  there is a unique element $\phi_1 \in \mathcal{H}(\eta K(0) \eta^{-1},\rho)$ such that $r'(\phi_1)=\phi$. Thus supp($\phi_1$)=$\mathfrak{P} w_1 \mathfrak{P}$ and $\phi_1^2= q^n+ (q^n-1)\phi_1$. Now $\phi_1$ can be extended to $G$ and viewed as an element in $\mathcal{H}(G, \rho)$ as $\mathfrak{P} w_1 \mathfrak{P} \subseteq \eta K(0) \eta^{-1}\subseteq G$. Thus $\phi_1$ satisfies the following relation in $\mathcal{H}(G,\rho)$:

\[
\phi_1^2= q^n+ (q^n-1)\phi_1.
\]\par 

Thus we have shown there exists $\phi_i \in \mathcal{H}(G,\rho)$ with supp($\phi_i$)=$\mathfrak{P}w_i\mathfrak{P}$ satisfying $\phi_i^2= q^n+ (q^n-1)\phi_i$  for $i=0,1$.

\begin{lemma}\label{lem_5}
	$\phi_0$ and $\phi_1$ are units in $\mathcal{H}(G,\rho)$.
\end{lemma}
\begin{proof}
	As $\phi_i^2= q^n + (q^n-1) \phi_i$ for $i=0,1$. So $\phi_i( \frac{\phi_i+(1-q^n)1}{q^n})= 1$ for i=0,1. Hence $\phi_0$ and $\phi_1$ are units in $\mathcal{H}(G,\rho)$.
\end{proof}\par

 \begin{lemma}\label{lem_6}
	Let $\phi, \psi \in \mathcal{H}(G,\rho) $ with support of $\phi, \psi$ being $\mathfrak{P}x\mathfrak{P}, \mathfrak{P}y\mathfrak{P}$ respectively. Then supp($\phi * \psi$)=supp($\phi \psi$) $\subseteq$ (supp($\phi$))(supp($\psi$))=$\mathfrak{P}x\mathfrak{P}y\mathfrak{P}$.
\end{lemma}
\begin{proof}
The proof is same as that of Lemma 5 in \cite{sandeep}.
\end{proof}\par 

Let $\zeta=w_0w_1$, So 
\begin{center}
	$\zeta = 
	\begin{bmatrix}
	\varpi_E Id_n & 0 &0\\
	0 & 1&0\\
	0     &0& \varpi_E^{-1} Id_n
	\end{bmatrix}.$
\end{center}\par

\begin{lemma}\label{lem_10}
	$\text{supp}(\phi_0* \phi_1)=\mathfrak{P} \zeta \mathfrak{P}=\mathfrak{P}w_0 w_1\mathfrak{P}$. 
\end{lemma}
\begin{proof} 
	It follows from Lemma \ref{lem_6} that $\text{supp}(\phi_0* \phi_1)\subseteq \mathfrak{P}w_0\mathfrak{P}w_1\mathfrak{P}$. Now let us recall $\mathfrak{P}_0, \mathfrak{P}_{+}, \mathfrak{P}_{-}$.
	
	\begin{center}
		$ \mathfrak{P}_{0}=
		\Bigg\lbrace
		\begin{bmatrix}
		a & 0 &0\\
		0& \lambda & 0\\ 
		0 & 0&{^t}{\overline{a}}{^{-1}}
		\end{bmatrix}\mid a \in \mathrm{GL}_n(\mathfrak{O}_E), \lambda \in \mathfrak{O}_E^{\times}, \lambda\overline{\lambda}=1 
		\Bigg \rbrace$,
		\end {center}
		
		\begin{center}
			$\mathfrak{P}_{+}= \Bigg \lbrace \begin{bmatrix}
			Id_n &u &X \\
			0 &1 &-{}^t\overline{u}\\
			0 &0 &Id_n
			\end{bmatrix} \mid  X \in \mathrm{M}_n(\mathfrak{O}_E), u \in \mathrm{M}_{n \times 1}(\mathfrak{O}_E), X + {^t}{\overline{X}+u{}^t\overline{u}}=0 \Bigg \rbrace,$
		\end{center}
		
		\begin{center}
			$\mathfrak{P}_{-}= \Bigg \lbrace \begin{bmatrix}
			Id_n &0 &0 \\
			-{}^t\overline{u} &1 &0\\
			X &u &Id_n
			\end{bmatrix} \mid  X \in \mathrm{M}_n(\mathbf{p}_E), u \in \mathrm{M}_{n \times 1}(\mathbf{p}_E), X + {^t}{\overline{X}+u{}^t\overline{u}}=0 \Bigg \rbrace.$
		\end{center}\par			
		
		It is easy observe that $w_0\mathfrak{P}_{-}w_0^{-1} \subseteq \mathfrak{P}_{+}, w_0\mathfrak{P}_{0}w_0^{-1}=\mathfrak{P}_{0}, w_1^{-1}\mathfrak{P}_{+}w_1 \subseteq \mathfrak{P}_{-}$. Now we have  
		
		\begin{align*}
		\mathfrak{P}w_0\mathfrak{P}w_1\mathfrak{P}&= \mathfrak{P}w_0\mathfrak{P}_{-}\mathfrak{P}_{0}\mathfrak{P}_{+}w_1\mathfrak{P}\\
		&=\mathfrak{P}w_0\mathfrak{P}_{-}w_0^{-1}w_0\mathfrak{P}_{0}w_0^{-1}w_0w_1w_1^{-1}\mathfrak{P}_{+}w_1\mathfrak{P}\\
		&\subseteq\mathfrak{P}\mathfrak{P}_{+}\mathfrak{P}_{0}w_0w_1\mathfrak{P}_{-}\mathfrak{P}\\
		&=\mathfrak{P}w_0w_1\mathfrak{P}\\
		&=\mathfrak{P}\zeta\mathfrak{P}.
		\end{align*}

		So $\mathfrak{P}w_0\mathfrak{P}w_1\mathfrak{P}\subseteq \mathfrak{P}w_0w_1\mathfrak{P}=\mathfrak{P}\zeta\mathfrak{P}$. On the contrary, as $1 \in \mathfrak{P}$, so $\mathfrak{P}\zeta\mathfrak{P}=\mathfrak{P}w_0w_1\mathfrak{P} \subseteq \mathfrak{P}w_0\mathfrak{P}w_1\mathfrak{P}$. Hence we have $\mathfrak{P}w_0\mathfrak{P}w_1\mathfrak{P}= \mathfrak{P}w_0w_1\mathfrak{P}=\mathfrak{P}\zeta\mathfrak{P}$. Therefore $\text{supp}(\phi_0*\phi_1) \subseteq \mathfrak{P}w_0\mathfrak{P}w_1\mathfrak{P}= \mathfrak{P}w_0w_1\mathfrak{P}=\mathfrak{P}\zeta\mathfrak{P}$. This implies $\text{supp}(\phi_0*\phi_1)= \varnothing$ or $\mathfrak{P}\zeta\mathfrak{P}$. But if $\text{supp}(\phi_0*\phi_1)= \varnothing$ then $(\phi_0*\phi_1)=0$ which is a contradiction. Thus $\text{supp}(\phi_0*\phi_1)=\mathfrak{P}\zeta\mathfrak{P}$. 
	\end{proof}

We shall now show that $\phi_0$ and $\phi_1$ generate the Hecke algebra $\mathcal{H}(G, \rho)$. Recall that $\mathcal{H}(G,\rho)$ consists of $\mathbb{C}$-linear combinations of maps $f \colon G \longrightarrow End_{\mathbb{C}}(V^{\vee})$ such that $f$ is supported on $\mathfrak{P}x\mathfrak{P}$ where $x \in \mathfrak{I}_G(\rho)$ and $f(pxp')= \rho^{\vee}(p)f(x)\rho^{\vee}(p')$ for $p,p' \in \mathfrak{P}$. Also note that $\mathfrak{I}_G(\rho)= \mathfrak{P}<\mathfrak{P}_0, w_0, w_1>\mathfrak{P}$. Observe that as $w_0$ normalizes $\mathfrak{P}_0$ and as $\mathfrak{P}_0$ being a subgroup of $\mathfrak{P}$, so $\mathcal{H}(G,\rho)$ consists of $\mathbb{C}$-linear combinations of maps $f \colon G \longrightarrow End_{\mathbb{C}}(V^{\vee})$ such that $f$ is supported on $\mathfrak{P}x\mathfrak{P}$ where $x$ is a word in $w_0, w_1$ with $w_0^2=w_1^2=Id$ and $f(pxp')= \rho^{\vee}(p)f(x)\rho^{\vee}(p')$ for $p,p' \in \mathfrak{P}$. Recall that support of $\phi_0$ is $\mathfrak{P}w_0\mathfrak{P}$ and support of $\phi_1$ is $\mathfrak{P}w_1\mathfrak{P}$. Also note that from Lemma \ref{lem_10}, we have $\text{supp}(\phi_0* \phi_1)=\mathfrak{P}w_0 w_1\mathfrak{P}$. So any $f \colon G \longrightarrow End_{\mathbb{C}}(V^{\vee})$ such that $f$ is supported on $\mathfrak{P}x\mathfrak{P}$ where $x$ is a word in $w_0, w_1$ and $f(pxp')= \rho^{\vee}(p)f(x)\rho^{\vee}(p')$ for $p,p' \in \mathfrak{P}$ can be written as word in $\phi_0$ and $\phi_1$. Therefore, $\mathcal{H}(G,\rho)$ consists of $\mathbb{C}$-linear combinations of words in $\phi_0$ and $\phi_1$. Hence, $\phi_0$ and $\phi_1$ generate $\mathcal{H}(G, \rho)$. Let us denote the Hecke algebra $\mathcal{H}(G, \rho)$ by $\mathcal{A}$. So we have
\[
\mathcal{A}= \mathcal{H}(G,\rho)= \left\langle \phi_i \colon G \to End_{\mathbb{C}}(\rho^{\vee})\; \middle| \;
\begin{varwidth}{\linewidth} 
$\phi_i$ is supported on 
$\mathfrak{P}w_i\mathfrak{P}$\\
and $\phi_i(pw_ip')= \rho^{\vee}(p)\phi_i(w_i)\rho^{\vee}(p')$\\
where $p,p' \in \mathfrak{P}, \, i=0,1$
\end{varwidth}
\right\rangle
\] where $\phi_i$ satisfies the relation:

\begin{center}
	$\phi_i^2= q^n+ (q^n-1) \phi_i$ for $i=0,1$.
\end{center}

As $\phi_0, \phi_1$ are units in algebra $\mathcal{A}$, so $\psi = \phi_0\phi_1$ is a unit too in $\mathcal{A}$ and $\psi^{-1}= \phi_1^{-1}\phi_0^{-1}$. Now as we have seen before that supp($\phi_0\phi_1$) $\subseteq \mathfrak{P}w_0w_1\mathfrak{P} \Longrightarrow supp(\psi) \subseteq \mathfrak{P}\zeta \mathfrak{P}\Longrightarrow supp(\psi)= \varnothing \,\text{or} \, \mathfrak{P}\zeta\mathfrak{P}$. If supp($\psi$)= $\varnothing \Longrightarrow \psi= 0$ which is a contradiction as $\psi$ is a unit in $\mathcal{A}$. So supp($\psi$) = $\mathfrak{P} \zeta \mathfrak{P}$. As $\psi$ is a unit in $\mathcal{A}$, we can show as before that supp($\psi^2$) = $\mathfrak{P} \zeta^2 \mathfrak{P}$. Hence by induction on $n \in \mathbb{N}$, we can further show that that supp($\psi^n$)= $\mathfrak{P} \zeta^n \mathfrak{P}$ for $n \in \mathbb{N}$. \par

Now $\mathcal{A}$ contains a sub- algebra generated by 
$\psi, \psi^{-1}$ over $\mathbb{C}$ and we denote this sub-algebra by $\mathcal{B}$. So $\mathcal{B}=\mathbb{C}[\psi,\psi^{-1}]$ where

\[
\mathcal{B}=\mathbb{C}[\psi,\psi^{-1}] = \left\lbrace c_k\psi^k + \cdots +c_l \psi^l \; \middle| \;
\begin{varwidth}{\linewidth} 
$c_k, \ldots ,c_l \in \mathbb{C};$\\
$k <l ; k,l \in \mathbb{Z}$
\end{varwidth}
\right\rbrace.
\]
\par

\begin{proposition}\label{algebra_isomorphism}
	The unique algebra homomorphism $\mathbb{C}[x, x^{-1}] \longrightarrow \mathcal{B}$ given by $x \longrightarrow \psi$ is an isomorphism. So $\mathcal{B} \simeq \mathbb{C}[x, x^{-1}]$. 
\end{proposition}
\begin{proof}
The proof is same as that of Proposition 18 in \cite{sandeep}. 
\end{proof}

\subsection{Ramified case:}

In this section we determine the structure of $\mathcal{H}(G,\rho)$ for the ramified case when $n$ is even. Recall $ \mathfrak{I}_G(\rho)= \mathfrak{P}N_G(\rho_0)\mathfrak{P}$. But from lemma \ref{lem_2}, $N_G(\rho_0)= \left<\mathfrak{P}_0, w_0, w_1\right>$. So $\mathfrak{I}_G(\rho)=\mathfrak{P}\left<\mathfrak{P}_0, w_0, w_1\right>\mathfrak{P}= \mathfrak{P}\left<w_0, w_1\right> \mathfrak{P}$, as $\mathfrak{P}_0$ is a subgroup of $\mathfrak{P}$. Let $V$ be the vector space corresponding to $\rho$. Let us recall that  $\mathcal{H}(G, \rho)$ consists of maps $f \colon G \to End_{\mathbb{C}}(V^{\vee})$ such that support of $f$ is compact and $f(pgp')= \rho^{\vee}(p)f(g)\rho^{\vee}(p')$ for $p,p' \in \mathfrak{P}, g \in G$. In fact $\mathcal{H}(G,\rho)$ consists of $\mathbb{C}$-linear combinations of maps $f \colon G \longrightarrow End_{\mathbb{C}}(V^{\vee})$ such that $f$ is supported on $\mathfrak{P}x\mathfrak{P}$ where $x \in \mathfrak{I}_G(\rho)$ and $f(pxp')= \rho^{\vee}(p)f(x)\rho^{\vee}(p')$ for $p,p' \in \mathfrak{P}$. We shall now show there exists $\phi_0 \in \mathcal{H}(G,\rho)$ with support $\mathfrak{P}w_0\mathfrak{P}$ and satisfies $\phi_0^2= q^{n/2}+ (q^{n/2}-1)\phi_0$. Let

\[
K(0)= \mathrm{U}(n,n+1) \cap \mathrm{GL}_{2n+1}(\mathfrak{O}_E)=\{g \in \mathrm{GL}_{2n+1}(\mathfrak{O}_E)\mid ^t\overline{g}Jg=J\},
\]
\[
K_1(0)=\{g \in Id_{2n+1}+ \varpi_E \mathrm{M}_{2n+1}(\mathfrak{O}_E)\mid ^t\overline{g}Jg=J\},
\]
\[
\mathsf{G}= \{g \in \mathrm{GL}_{2n+1}(k_E) \mid ^t\overline{g}Jg=J\}.
\]

The map $r$ from $K(0)$ to $\mathsf{G}$ given by $r\colon K(0)\xrightarrow{\text{mod} \, p_E}\mathsf{G}$ is a surjective group homomorphism with kernel $K_1(0)$. So by the first isomorphism theorem of groups we have:

\begin{center}
	$\frac{K(0)}{K_1(0)}\cong \mathsf{G}$.
\end{center}

$r(\mathfrak{P})= \mathsf{P}=
\begin{bmatrix}
\mathrm{GL}_n(k_E) &M_{n \times 1}(k_E) &\mathrm{M}_n(k_E) \\
0 & \mathrm{U}_1(k_E) & M_{1 \times n}(k_E) \\
0         & 0& \mathrm{GL}_n(k_E)
\end{bmatrix} \bigcap \mathsf{G}$= Siegel parabolic subgroup of $\mathsf{G}$.\\

Now $\mathsf{P}= \mathsf{L} \ltimes \mathsf{U}$, where $\mathsf{L}$ is the Siegel Levi component of $\mathsf{P}$ and $\mathsf{U}$ is the unipotent radical of $\mathsf{P}$. Here

\[
\mathsf{L}= \Bigg \lbrace
\begin{bmatrix}
a & 0&0\\
0 &\lambda &0\\
0 &0& ^t\overline{a}^{-1}
\end{bmatrix} \mid a \in \mathrm{GL}_n(k_E), \lambda \in E^{\times}, \lambda \overline{\lambda}=1 \Bigg \rbrace,
\]

\[
\mathsf{U}= \Bigg \lbrace \begin{bmatrix}
Id_n &u &X \\
0 &1 &-{}^t\overline{u}\\
0 &0 &Id_n
\end{bmatrix} \mid  X \in \mathrm{M}_n(k_E), u \in \mathrm{M}_{n \times 1}(k_E), X + {^t}{\overline{X}+u{}^t\overline{u}}=0 \Bigg \rbrace.
\]

Let $V$ be the vector space corresponding to $\rho$. The Hecke algebra $\mathcal{H}(K(0),\rho)$ is a sub-algebra of $\mathcal{H}(G,\rho)$.\par

Let $\overline{\rho}$ be the representation of $\mathsf{P}$ which when inflated to $\mathfrak{P}$ is given by $\rho$ and $V$ is also the vector space corresponding to $\overline{\rho}$. Recall the Hecke algebra $\mathcal{H}(\mathsf{G},\overline{\rho})$ has the same structure as was defined earlier in section 6.1 for the unramified case.\par

Now the homomorphism $r \colon K(0) \longrightarrow \mathsf{G}$ extends to a map from  $\mathcal{H}(K(0),\rho)$ to $\mathcal{H}(\mathsf{G},\overline{\rho})$ which we again denote by $r$. Thus $r \colon \mathcal{H}(K(0),\rho) \longrightarrow \mathcal{H}(\mathsf{G},\overline{\rho})$ is given by

\[
r(\phi)(r(x))=\phi(x)
\]
\[ 
\text{for} \, \phi \in \mathcal{H}(K(0),\rho) \, \text{and} \, x \in K(0).
\]

As in the unramified case, when $n$ is odd, we can show that $\mathcal{H}(K(0),\rho)$ is isomorphic to $\mathcal{H}(\mathsf{G},\overline{\rho})$ as algebras via $r$.\par

Let $w=r(w_0)=r(
\begin{bmatrix}
0 &0& Id_n\\
0&1 & 0\\
Id_n &0 &0
\end{bmatrix})=
\begin{bmatrix}
0 &0& Id_n \\
0&1 & 0\\
Id_n &0 &0
\end{bmatrix} \in \mathsf{G}$. Clearly $K(0) \supseteq \mathfrak{P}\amalg\mathfrak{P}w_0\mathfrak{P}$ and $\mathsf{G}\supseteq \mathsf{P}\amalg \mathsf{P}w\mathsf{P}$.\par

Now $\mathsf{G}$ is a finite group. In fact, it is the special orthogonal group consisting of matrices of size $(2n+1) \times (2n+1)$ over finite field $k_E$ or $\mathbb{F}_q$. So $\mathsf{G} = SO_{2n+1}(\mathbb{F}_q)$. \par 

According to the Theorem 6.3 in \cite{MR2276353}, there exists a unique $\phi$ in  $\mathcal{H}(\mathsf{G}, \overline{\rho})$ with support $\mathsf{P} w \mathsf{P}$ such that $\phi^2= q^{n/2}+ (q^{n/2}-1)\phi$. Hence there is a unique element $\phi_0 \in \mathcal{H}(K(0),\rho)$ such that $r(\phi_0)=\phi$. Thus supp($\phi_0$)=$\mathfrak{P} w_0 \mathfrak{P}$ and $\phi_0^2= q^{n/2}+ (q^{n/2}-1)\phi_0$. Now $\phi_0$ can be extended to $G$ and viewed as an element in $\mathcal{H}(G, \rho)$ as $\mathfrak{P} w_0 \mathfrak{P} \subseteq K(0) \subseteq G$. Thus $\phi_0$ satisfies the following relation in $\mathcal{H}(G,\rho)$:

\[
\phi_0^2= q^{n/2}+ (q^{n/2}-1)\phi_0.
\]\par 

We shall now show there exists $\phi_1 \in \mathcal{H}(G, \rho)$ with support $\mathfrak{P}w_1\mathfrak{P}$ satisfying the same relation as $\phi_0$.\par

Recall that $w_1=
\begin{bmatrix}
0 & 0& \overline\varpi^{-1}_E Id_n \\
0&1&0\\
\varpi_E Id_n & 0&0
\end{bmatrix}, \overline\varpi^{-1}_E= -\varpi_E^{-1}$. 
So $w_1= 
\begin{bmatrix}
0 & 0& -\varpi^{-1}_E Id_n \\
0&1&0\\
\varpi_E Id_n & 0&0
\end{bmatrix}$. Let $\eta \in \mathrm{U}(n,n+1)$ be such that 
 $\eta w_1 \eta^{-1}= J'=
\begin{bmatrix}
0 &0& -Id_n\\
0&1&0\\
Id_n & 0&0
\end{bmatrix}$ and 
\[ 
\eta 
\begin{bmatrix}
\mathrm{GL}_n(\mathfrak{O}_E)    & \mathrm{M}_{n \times 1}(\mathfrak{O}_E) & \mathrm{M}_n(\mathfrak{O}_E) \\
\mathrm{M}_{1 \times n}(\mathbf{p}_E) & \mathrm{U}_1(\mathfrak{O}_E) & \mathrm{M}_{1 \times n}(\mathfrak{O}_E)\\
\mathrm{M}_n(\mathbf{p}_E)  & \mathrm{M}_{n \times 1}(\mathbf{p}_E) & \mathrm{GL}_n(\mathfrak{O}_E) 
\end{bmatrix} \eta^{-1}= \begin{bmatrix}
\mathrm{GL}_n(\mathfrak{O}_E)    & \mathrm{M}_{n \times 1}(\mathbf{p}_E) & \mathrm{M}_n(\mathbf{p}_E) \\
\mathrm{M}_{1 \times n}(\mathfrak{O}_E) & \mathrm{U}_1(\mathfrak{O}_E) & \mathrm{M}_{1 \times n}(\mathbf{p}_E)\\
\mathrm{M}_n(\mathfrak{O}_E)  & \mathrm{M}_{n \times 1}(\mathfrak{O}_E) & \mathrm{GL}_n(\mathfrak{O}_E) 
\end{bmatrix}.
\] Recall that $\mathfrak{P}$ looks as follows:

\begin{center}
	$\mathfrak{P}=\begin{bmatrix}
	\mathrm{GL}_n(\mathfrak{O}_E)    & \mathrm{M}_{n \times 1}(\mathfrak{O}_E) & \mathrm{M}_n(\mathfrak{O}_E) \\
	\mathrm{M}_{1 \times n}(\mathbf{p}_E) & \mathrm{U}_1(\mathfrak{O}_E) & \mathrm{M}_{1 \times n}(\mathfrak{O}_E)\\
	\mathrm{M}_n(\mathbf{p}_E)  & \mathrm{M}_{n \times 1}(\mathbf{p}_E) & \mathrm{GL}_n(\mathfrak{O}_E) 
	\end{bmatrix} \bigcap G.$
\end{center}\par

Note that 

\[\eta G \eta^{-1}= \lbrace g \in \mathrm{GL}_{2n+1}(E) \mid ^t \overline{g} J' g= J'\rbrace.\]

Hence
\begin{center}
	$\eta\mathfrak{P}\eta^{-1}=\begin{bmatrix}
\mathrm{GL}_n(\mathfrak{O}_E)    & \mathrm{M}_{n \times 1}(\mathbf{p}_E) & \mathrm{M}_n(\mathbf{p}_E) \\
\mathrm{M}_{1 \times n}(\mathfrak{O}_E) & \mathrm{U}_1(\mathfrak{O}_E) & \mathrm{M}_{1 \times n}(\mathbf{p}_E)\\
\mathrm{M}_n(\mathfrak{O}_E)  & \mathrm{M}_{n \times 1}(\mathfrak{O}_E) & \mathrm{GL}_n(\mathfrak{O}_E) 
	\end{bmatrix} \bigcap \eta G \eta^{-1}.$
\end{center}\par

Therefore $\eta\mathfrak{P}\eta^{-1}$ is the opposite of the Siegel Parahoric subgroup of $\eta G \eta^{-1}$. Let

\begin{center}
	$K'(0)= \langle\mathfrak{P}, w_1\rangle.$
\end{center}

And let
\begin{align*}
\mathsf{G}' &=\lbrace g \in \mathrm{GL}_{2n+1}(k_E) \mid {}^t\overline{g}J'g=J'\rbrace\\
&= \lbrace g \in \mathrm{GL}_{2n+1}(k_E) \mid {}^tgJ'g=J'\rbrace.
\end{align*} 

Let $r'\colon K'(0) \longrightarrow \mathsf{G}'$ be the group homomorphism given by
\[r'(x)= (\eta x \eta^{-1})mod p_E \,\text{where} \, x \in K'(0).\]

So we have $r'(K(0))= (\eta K'(0) \eta^{-1}) mod p_E= (\eta\langle \mathfrak{P}, w_1 \rangle\eta^{-1})mod p_E$. Let \[r'(\mathfrak{P})=(\eta \mathfrak{P} \eta^{-1})mod p_E = \overline{\mathsf{P}}'.\] We can see that  $r'(w_1)= (\eta w_1 \eta^{-1}) mod p_E=J' mod p_E=w'=\begin{bmatrix}
0 &0& -Id_n\\
0&1 & 0\\
Id_n&0&0
\end{bmatrix}$. So $\overline {\mathsf{P}}'= r'(\mathfrak{P})= (\eta \mathfrak{P} \eta^{-1}) mod p_E=
\begin{bmatrix}
\mathrm{GL}_n(k_E) &0 &0 \\
M_{1 \times n}(k_E) & \mathrm{U}_1(k_E) & 0 \\
\mathrm{M}_n(k_E)   & M_{n \times 1}(k_E)& \mathrm{GL}_n(k_E)
\end{bmatrix} \bigcap \mathsf{G}'$. Clearly $\overline {\mathsf{P}}'$ is the opposite of Siegel parabolic subgroup of $\mathsf{G}'$. Hence  $r'(K'(0))= \langle \overline {\mathsf{P}}', w' \rangle= \mathsf{G}'$, as $\overline {\mathsf{P}}'$ is a maximal subgroup of $\mathsf{G}'$ and $w'$ does not lie in $\overline {\mathsf{P}}'$. So $r'$ is a surjective homomorphism of groups.\par

Let $V$ be the vector space corresponding to $\rho$. The Hecke algebra $\mathcal{H}(K'(0),\rho)$ is a sub-algebra of $\mathcal{H}(G,\rho)$.\par

Let $\overline{\rho}'$ be the representation of $\overline{\mathsf{P}}'$ which when inflated to $^{\eta}\mathfrak{P}$ is given by $^\eta{\rho}$ and $V$ is also the vector space corresponding to $\overline{\rho}'$. Note that the Hecke algebra $\mathcal{H}(\mathsf{G}',\overline{\rho}')$ has a similar structure as that of
$\mathcal{H}(\mathsf{G},\overline{\rho})$ which was defined earlier.\par

Now the homomorphism $r' \colon K'(0) \longrightarrow \mathsf{G'}$ extends to a map from  $\mathcal{H}(K'(0),\rho)$ to $\mathcal{H}(\mathsf{G}',\overline{\rho}')$ which we again denote by $r'$. Thus $r' \colon \mathcal{H}(K'(0),\rho) \longrightarrow \mathcal{H}(\mathsf{G}',\overline{\rho}')$ is given by

\[
r'(\phi)(r'(x))=\phi(x)
\]
\[ 
\text{for} \, \phi \in \mathcal{H}(K'(0),\rho) \, \text{and} \, x \in K'(0).
\]

As in the unramified case when $n$ is odd, we can show that $\mathcal{H}(K'(0),\rho)$ is isomorphic to $\mathcal{H}(\mathsf{G}',\overline{\rho}')$ as algebras via $r'$.\par

Clearly, $K'(0) \supseteq \mathfrak{P}\amalg\mathfrak{P}w_1\mathfrak{P}$ and  $\mathsf{G}' \supseteq \overline{\mathsf{P}}'\amalg 
\overline{\mathsf{P}}'w'\overline{\mathsf{P}}'$.\par

Now $\mathsf{G}'$ is a finite group over the field $K_E$ or $\mathbb{F}_q$. Note that $\mathsf{G}' \cong Sp_{2n}(k_E)$. According to the Theorem 6.3 in \cite{MR2276353}, there exists a unique $\phi$ in  $\mathcal{H}(\mathsf{G}', \overline{\rho}')$ with support $\overline{\mathsf{P}}' w' \overline{\mathsf{P}}'$ such that $\phi^2= q^{n/2}+ (q^{n/2}-1)\phi$. Hence there is a unique element $\phi_1 \in \mathcal{H}(K'(0),\rho)$ such that $r'(\phi_1)=\phi$. Thus supp($\phi_1$)=$\mathfrak{P} w_1 \mathfrak{P}$ and $\phi_1^2= q^{n/2}+ (q^{n/2}-1)\phi_1$. Now $\phi_1$ can be extended to $G$ and viewed as an element in $\mathcal{H}(G, \rho)$ as $\mathfrak{P} w_1 \mathfrak{P} \subseteq K'(0) \subseteq G$. Thus $\phi_1$ satisfies the following relation in $\mathcal{H}(G,\rho)$:

\[
\phi_1^2= q^{n/2}+ (q^{n/2}-1)\phi_1.
\]\par 

Thus we have shown there exists $\phi_i \in \mathcal{H}(G,\rho)$ with supp($\phi_i$)=$\mathfrak{P}w_i\mathfrak{P}$ satisfying $\phi_i^2= q^{n/2}+ (q^{n/2}-1)\phi_i$  for $i=0,1$. It can be further shown as in the unramified case that $\phi_0$ and $\phi_1$ generate the Hecke algebra $\mathcal{H}(G, \rho)$. Let us denote the Hecke algebra $\mathcal{H}(G, \rho)$ by $\mathcal{A}$. So 

\[
\mathcal{A}= \mathcal{H}(G,\rho)= \left\langle \phi_i \colon G \to End_{\mathbb{C}}(\rho^{\vee})\; \middle| \;
\begin{varwidth}{\linewidth} 
$\phi_i$ is supported on 
$\mathfrak{P}w_i\mathfrak{P}$\\
 $\,\text{and} \, \phi_i(pw_ip')= \rho^{\vee}(p)\phi_i(w_i)\rho^{\vee}(p')$\\
where $p,p' \in \mathfrak{P}, \, i=0,1$
\end{varwidth}
\right\rangle
\] where $\phi_i$ has support $\mathfrak{P} w_i \mathfrak{P}$  and $\phi_i$ satisfies the relation:

\begin{center}
	$\phi_i^2= q^{n/2}+ (q^{n/2}-1) \phi_i$ for $i=0,1$.
\end{center}  

\begin{lemma}\label{lem_7}
	$\phi_0$ and $\phi_1$ are units in $\mathcal{A}$.
\end{lemma}
\begin{proof}
	As $\phi_i^2= q^{n/2}+ (q^{n/2}-1) \phi_i$ for $i=0,1$. So $\phi_i( \frac{\phi_i+(1-q^{n/2})1}{q^{n/2}})= 1$ for i=0,1. Hence $\phi_0$ and $\phi_1$ are units in $\mathcal{A}$.
\end{proof}\par

As $\phi_0, \phi_1$ are units in $\mathcal{A}$ which is an algebra, so $\psi= \phi_0\phi_1$ is a unit too in $\mathcal{A}$ and $\psi^{-1}= \phi_1^{-1}\phi_0^{-1}$. As in the unramified case when $n$ is odd, we can show that $\mathcal{A}$ contains sub-algebra $\mathcal{B}=\mathbb{C}[\psi,\psi^{-1}]$ where

\[
\mathcal{B}=\mathbb{C}[\psi,\psi^{-1}] = \left\lbrace c_k\psi^k + \cdots +c_l \psi^l \; \middle| \;
\begin{varwidth}{\linewidth} 
$c_k, \ldots ,c_l \in \mathbb{C};$\\
$k <l ; k,l \in \mathbb{Z}$
\end{varwidth}
\right\rbrace.
\]
\par 

Further, as in the unramified case when $n$ is odd, we can show that $\mathbb{C}[\psi,\psi^{-1}] \simeq \mathbb{C}[x,x^{-1}]$ as $\mathbb{C}$-algebras. \par

\section{Structure of $\mathcal{H}(L,\rho_0)$} \label{sec_7}

In this section we describe the structure of $\mathcal{H}(L,\rho_0)$. Thus we need first to determine 

\[
N_L(\rho_0)=\lbrace m \in N_L(\mathfrak{P}_0) \mid \rho_0^m \simeq \rho_0 \rbrace. 
\] \par

We know from lemma \ref{Normalizer_of_K_0_in_GL_n(E)} that $N_{\mathrm{GL}_n(E)}(K_0)= K_0Z$, so we have $N_L(\mathfrak{P}_0)= Z(L)\mathfrak{P}_0$. Since $Z(L)$ clearly normalizes $\rho_0$  and $\rho_0$ is an irreducible cuspidal representation of $\mathfrak{P}_0$, so $N_L(\rho_0)= Z(L)\mathfrak{P}_0= \coprod_{n \in \mathbb{Z}}\mathfrak{P}_0 \zeta^n$.\par

Define $\alpha \in \mathcal{H}(L,\rho_0)$ by $\text{supp}(\alpha)= \mathfrak{P}_0\zeta$ and $\alpha(\zeta)=1_{V^{\vee}}$. We can show that $\alpha^n(\zeta^n) = (\alpha(\zeta))^n$ for $n \in \mathbb{Z}$ and $\text{supp}(\alpha^n)= \mathfrak{P}_0 \zeta^n \mathfrak{P}_0= \mathfrak{P}_0 \zeta^n =  \zeta^n \mathfrak{P}_0$ for $n \in \mathbb{Z}$. Further we can show that  $\mathcal{H}(L,\rho_0)=\mathbb{C}[\alpha, \alpha^{-1}]$. For details refer to section 7 in \cite{sandeep}.\par 

\begin{proposition}\label{algebra_isomorphism_1}
	The unique algebra homomorphism $\mathbb{C}[x, x^{-1}] \longrightarrow \mathbb{C}[\alpha, \alpha^{-1}]$ given by $x \longrightarrow \alpha$ is an isomorphism. So $\mathbb{C}[\alpha, \alpha^{-1}] \simeq \mathbb{C}[x, x^{-1}]$. 
\end{proposition}

We have already shown before in sections 6.1 and 6.2 that $\mathcal{B}= \mathbb{C}[\psi,\psi^{-1}]$ is a sub-algebra of $\mathcal{A}= \mathcal{H}(G,\rho)$, where $\psi$ is supported on $\mathfrak{P}\zeta\mathfrak{P}$ and $\mathcal{B} \cong \mathbb{C}[x,x^{-1}]$. As $\mathcal{H}(L,\rho_0)= \mathbb{C}[\alpha,\alpha^{-1}]\cong \mathbb{C}[x,x^{-1}]$, so  $\mathcal{B} \cong \mathcal{H}(L,\rho_0)$ as $\mathbb{C}$-algebras. Hence $\mathcal{H}(L,\rho_0)$ can be viewed as a sub-algebra of $\mathcal{H}(G,\rho)$.\par 

Now we would like to find out how simple $\mathcal{H}(L,\rho_0)$-modules look like. Thus to understand them we need to find out how simple $\mathbb{C}[x,x^{-1}]$-modules look like.

\section{Calculation of simple $\mathcal{H}(L,\rho_0)$-modules} \label{sec_8}

Recall that $\mathcal{H}(L,\rho_0) = \mathbb{C}[\alpha, \alpha^{-1}]$. Note that $\mathbb{C}[\alpha, \alpha^{-1}] \cong \mathbb{C}[x,x^{-1}]$ as $\mathbb{C}$-algebras. It can be shown by direct calculation that the simple $\mathbb{C}[x,x^{-1}]$-modules are of the form $\mathbb{C}_{\lambda}$ for $\lambda \in \mathbb{C}^{\times}$, where $\mathbb{C}_{\lambda}$ is the vector space $\mathbb{C}$ with the $\mathbb{C}[x,x^{-1}]$-module structure given by $x.z=\lambda z$ for $z \in \mathbb{C}_{\lambda}$. \par

So the distinct simple $\mathcal{H}(L,\rho_0)$-modules(up to isomorphism)  are the various $\mathbb{C}_{\lambda}$ for $\lambda \in \mathbb{C}^{\times}$. The module structure is determined by $\alpha.z = \lambda z$ for $z \in \mathbb{C}_{\lambda}$.

\section{Final calculations to answer the question} \label{sec_9}

\subsection{Calculation of $\delta_P(\zeta)$}

Let us recall the modulus character $\delta_P \colon P \longrightarrow \mathbb{R}^{\times}_{>0}$ introduced in section 1. The character $\delta_P$ is given by $\delta_P(p)=\|det(Ad \, p)|_{\mathrm{Lie} \, U}\|_F$ for $p \in P$, where $\mathrm{Lie} \, U$ is the Lie algebra of $U$. We have

\[
U= \Bigg\lbrace
\begin{bmatrix} 
Id_n &u &X\\
0  &1 & -{}^t\overline{u}\\
0&0&Id_n\\
\end{bmatrix}
\mid X \in \mathrm{M}_n(E), u \in \mathrm{M}_{n \times 1}(E), X+^t{}\overline{X}+u{}^t\overline{u}=0
\Bigg\rbrace,
\]

\[
\mathrm{Lie} \, U= \Bigg\lbrace
\begin{bmatrix} 
0 &u &X\\
0  &0 & -{}^t\overline{u}\\
0&0&0\\
\end{bmatrix}
\mid X \in \mathrm{M}_n(E), u \in \mathrm{M}_{n \times 1}(E), X+^t{}\overline{X}=0
\Bigg\rbrace.
\]

\subsubsection{Unramified case:}

Recall $\zeta=
\begin{bmatrix}
\varpi_E Id_n & 0 &0\\
0&1&0\\
0 &0&\varpi^{-1}_E Id_n
\end{bmatrix}$ in the unramified case. So

\[
(Ad \, \zeta)
\begin{bmatrix} 
Id_n &u &X\\
0  &1 & -{}^t\overline{u}\\
0&0&Id_n
\end{bmatrix}= \zeta 
\begin{bmatrix} 
Id_n &u &X\\
0  &1 & -{}^t\overline{u}\\
0&0&Id_n
\end{bmatrix}\zeta^{-1}=
\begin{bmatrix}
Id_n& \varpi_E u & \varpi_E^2 X \\
0 &1& -\varpi_E {}^t \overline{u}\\
0&0&Id_n
\end{bmatrix}.
\]
Hence
\begin{align*}
\delta_P(\zeta)&=\|det(Ad \, \zeta)|_{\mathrm{Lie} \, U}\|_F\\
&= \| -\varpi_E^{2n+2n^2}\|_F\\
&=\|-\varpi_F^{2n+2n^2} \|_F\\
&= q^{-2n-2n^2}.
\end{align*}

\subsubsection{Ramified case:}

Recall $\zeta=
\begin{bmatrix}
\varpi_E Id_n & 0 &0\\
0&1&0\\
0 &0& -\varpi^{-1}_E Id_n
\end{bmatrix}$ in the ramified case. So 

\[
(Ad \, \zeta)
\begin{bmatrix}
Id_n &u &X\\
0  &1 & -{}^t\overline{u}\\
0&0&Id_n
\end{bmatrix}= \zeta 
\begin{bmatrix}
Id_n &u &X\\
0  &1 & -{}^t\overline{u}\\
0&0&Id_n
\end{bmatrix}\zeta^{-1}=
\begin{bmatrix}
Id_n& \varpi_E u & -\varpi_E^2 X \\
0 &1& \varpi_E {}^t \overline{u}\\
0&0&Id_n
\end{bmatrix}.
\]
Hence
\begin{align*}
\delta_P(\zeta)&=\|det(Ad \, \zeta)|_{\mathrm{Lie}\, U}\|_F\\
&=\|\varpi_E^{2n+2n^2} \|_F\\
&=\|\varpi_F^{n+n^2} \|_F\\
&= q^{-n-n^2}.
\end{align*}

\subsection{Calculation of $(\phi_0*\phi_1)(\zeta)$}
In this section we calculate $(\phi_0*\phi_1)(\zeta)$. Let $g_i= q^{-n/2}\phi_i$ for $i=0,1$ in the unramified case and $g_i= q^{-n/4}\phi_i$ for $i=0,1$ in the ramified case. Determining $(\phi_0*\phi_1)(\zeta)$ would be useful in showing $g_0*g_1= T_P(\alpha)$ in both ramified and unramified cases. From now on, we assume without loss of generality that $\text{vol}\mathfrak{P}_0= \text{vol}\mathfrak{P}_{-}= \text{vol}\mathfrak{P}_{+}=1$. Thus we have $\text{vol}\mathfrak{P}=1$.\par
 
For $r \in \mathbb{Z}$ let,
	
	\begin{center}
		$K_{-,r}=
		\Bigg \lbrace \begin{bmatrix}
		Id_n &0 &0 \\
		-{}^t\overline{u} &1 &0\\
		X &u &Id_n
		\end{bmatrix} \mid  X \in \mathrm{M}_n(\mathbf{p}^r_E), u \in \mathrm{M}_{n \times 1}(\mathbf{p}^r_E), X + {^t}{\overline{X}+u{}^t\overline{u}}=0 \Bigg \rbrace.$	
	\end{center}

	\begin{center}
		$K_{+,r}=\Bigg \lbrace \begin{bmatrix}
		Id_n &u &X \\
		0 &1 &-{}^t\overline{u}\\
		0 &0 &Id_n
		\end{bmatrix} \mid  X \in \mathrm{M}_n(\mathfrak{O}^r_E), u \in \mathrm{M}_{n \times 1}(\mathfrak{O}^r_E), X + {^t}{\overline{X}+u{}^t\overline{u}}=0 \Bigg \rbrace,$
		
	\end{center}\par
	\begin{proposition}\label{prop_35}
		$(\phi_0*\phi_1)(\zeta)=\phi_0(w_0)\phi_1(w_1).$	
	\end{proposition}
	\begin{proof}			
		From Lemma \ref{lem_10}, $\text{supp}(\phi_0* \phi_1)=\mathfrak{P} \zeta \mathfrak{P}=\mathfrak{P}w_0 w_1\mathfrak{P}$. So now let us consider
		\begin{align*}
		(\phi_0*\phi_1)(\zeta)&=(\phi_0*\phi_1)(w_0w_1)\\
		&=\int_G \phi_0(y)\phi_1(y^{-1}\zeta)dy\\
		&=\int_{\mathfrak{P}w_0\mathfrak{P}}\phi_0(y)\phi_1(y^{-1}\zeta)dy.
		\end{align*}
		
		We know that $\mathfrak{P}w_0\mathfrak{P}= \underset{z \in {\mathfrak{P}w_0\mathfrak{P}/\mathfrak{P}}}{\amalg}z\mathfrak{P}$. Let $y=zp \in z\mathfrak{P}$. So we have
		\begin{align*}
		\phi_0(y)\phi_1(y^{-1}\zeta)&=\phi_0(zp)\phi_1(p^{-1}z^{-1}\zeta)\\
		&= \phi_0(z)\rho^{\vee}(p)\rho^{\vee}(p^{-1})\phi_1(z^{-1}\zeta)\\
		&=\phi_0(z)\phi_1(z^{-1}\zeta).
		\end{align*}				
		Hence \[(\phi_0*\phi_1)(\zeta)= \underset{z \in {\mathfrak{P}w_0\mathfrak{P}/\mathfrak{P}}}{\sum}\phi_0(z)\phi_1(z^{-1}\zeta) \text{Vol}{\mathfrak{P}}=\underset{z \in {\mathfrak{P}w_0\mathfrak{P}/\mathfrak{P}}}{\sum}\phi_0(z)\phi_1(z^{-1}\zeta)\] 
		
		Let $\alpha \colon \mathfrak{P}/w_0\mathfrak{P}w_0^{-1} \cap \mathfrak{P} \longrightarrow \mathfrak{P}w_0\mathfrak{P}/\mathfrak{P}$ be the map given by $\alpha(x(w_0\mathfrak{P}w_0^{-1} \cap \mathfrak{P}))= xw_0\mathfrak{P}$ where $x \in \mathfrak{P}$. We can observe that the map $\alpha$ is bijective. So $\mathfrak{P}/w_0\mathfrak{P}w_0^{-1} \cap \mathfrak{P}$ is in bijection with $\mathfrak{P}w_0\mathfrak{P}/\mathfrak{P}$.\par
		
		Hence \[(\phi_0*\phi_1)(\zeta)=  \underset{x \in {\mathfrak{P}/w_0\mathfrak{P}w_0^{-1} \cap \mathfrak{P}}}{\sum}\phi_0(xw_0)\phi_1(w_0^{-1}x^{-1}\zeta).\]
		From Iwahori factorization of $\mathfrak{P}$ we have $\mathfrak{P}=\mathfrak{P}_{-}\mathfrak{P}_{0}\mathfrak{P}_{+}= K_{-,1}\mathfrak{P}_{0}K_{+,0}$. Therefore $w_0\mathfrak{P}w_0^{-1}= ^{w_0}\mathfrak{P}=^{w_0}K_{-,1} ^{w_0}\mathfrak{P}_{0} ^{w_0}K_{+,0}= K_{+,1}\mathfrak{P}_0K_{-,0}$. So $\mathfrak{P}_0 \cap w_0\mathfrak{P}w_0^{-1}= \mathfrak{P} \cap ^{w_0}\mathfrak{P}= K_{+,1}\mathfrak{P}_0K_{-,1}$. Let $\beta \colon \mathfrak{P}/w_0\mathfrak{P}w_0^{-1} \cap \mathfrak{P} \longrightarrow K_{+,0}/K_{+,1}$ be the map given by $\beta(x(\mathfrak{P}\cap ^{w_0}\mathfrak{P}))= x_{+}K_{+,1}$ where $x \in \mathfrak{P}$ and $x=x_{+}px_{-}, x_{+} \in \mathfrak{P}_{+}, p \in \mathfrak{P}_{0}, x_{-} \in \mathfrak{P}_{-}$. We can observe that the map $\beta$ is bijective. So $\mathfrak{P}/w_0\mathfrak{P}w_0^{-1} \cap \mathfrak{P}$ is in bijection with $K_{+,0}/K_{+,1}$.\par
		
		Therefore 
		\begin{align*}
		(\phi_0*\phi_1)(\zeta)&=\underset{x_{+} \in {K_{+,0}/K_{+,1}}}{\sum}\phi_0(x_{+}w_0)\phi_1(w_0^{-1}x_{+}^{-1}\zeta)\\ &=\underset{x_{+}\in{K_{+,0}/K_{+,1}}}{\sum}\rho^{\vee}(x_{+})\phi_0(w_0)\phi_1(w_0^{-1}x_{+}^{-1}\zeta).
		\end{align*}
		As $\rho^{\vee}$ is trivial on $\mathfrak{P}_{+}$ and $x_{+} \in \mathfrak{P}_{+}$ so we have
		\begin{center}
			$(\phi_0*\phi_1)(\zeta)= \underset{x_{+} \in {K_{+,0}/K_{+,1}}}{\sum}\phi_0(w_0)\phi_1(w_0^{-1}x_{+}^{-1}\zeta).$
		\end{center}				
		
		The terms in above summation which do not vanish are the ones for which $w_0^{-1}x_{+}^{-1}\zeta \in \mathfrak{P}w_1\mathfrak{P} \Longrightarrow x_{+}^{-1} \in w_0\mathfrak{P}w_1\mathfrak{P}\zeta^{-1}\Longrightarrow x_{+} \in \zeta\mathfrak{P}w_1^{-1}\mathfrak{P}w_0^{-1} \Longrightarrow w_0^{-1}x_{+}w_0 \in w_1\mathfrak{P}w_1^{-1}\mathfrak{P}$. It is clear $w_1\mathfrak{P}w_1^{-1}\mathfrak{P}= (^{w_1}\mathfrak{P})(\mathfrak{P})$. As $^{w_1}\mathfrak{P}={}^{w_1}K_{-,1}^{w_1}\mathfrak{P}_0^{w_1}K_{+,0}= K_{-,2}\mathfrak{P}_0K_{+,-1}$, so $w_1\mathfrak{P}w_1^{-1}\mathfrak{P}= (^{w_1}\mathfrak{P})(\mathfrak{P})=   K_{-,2}\mathfrak{P}_0K_{+,-1}\mathfrak{P}_0K_{-,1}$. Hence we have $w_0^{-1}x_{+}w_0\in K_{-,2}\mathfrak{P}_0K_{+,-1}\mathfrak{P}_0K_{-,1} \Longrightarrow w_0^{-1}x_{+}w_0=k_{-}p_{0}k_{+}k'_{-}$ where $k_{-} \in K_{-,2}, k_{+} \in K_{+,-1}, k_{-}^{'}\in K_{-,1}, p_0 \in \mathfrak{P}_0$. Hence we have $p_{0}k_{+}= k_{-}^{-1}w_0^{-1}x_{+}w_0k_{-}^{'-1}$. Now as $w_0^{-1}x_{+}w_0 \in K_{-,0}, k_{-}^{-1} \in K_{-,2}, k_{-}^{'-1}\in K_{-,1}$, so $k_{-}^{-1}w_0^{-1}x_{+}w_0k_{-}^{'-1} \in K_{-,0}$ and $p_{0}k_{+} \in \mathfrak{P}_{0}K_{+,-1}$. But we know that $K_{-,0} \cap \mathfrak{P}_{0}K_{+,-1}= 1 \Longrightarrow p_{0}k_{+}= 1 \Longrightarrow w_0^{-1}x_{+}w_0=k_{-}k_{-}' \in K_{-,1} \Longrightarrow x_{+} \in w_0K_{-,1}w_0^{-1}= K_{+,1}$. As $x_{+} \in K_{+,1}$, so only the trivial coset contributes to the above summation. Hence \[(\phi_0*\phi_1)(\zeta)=\phi_0(w_0)\phi_1(w_0^{-1}\zeta)=\phi_0(w_0)\phi_1(w_1).\]
	\end{proof}
	
	\subsection{Relation between $g_0, g_1$ and $T_P(\alpha)$}
	
	\subsubsection{Unramified case:}
	
	Recall that $\mathcal{H}(G,\rho)=\langle \phi_0, \phi_1 \rangle$ where $\phi_0$ is supported on $\mathfrak{P}w_0\mathfrak{P}$ and $\phi_1$ is supported on $\mathfrak{P}w_1\mathfrak{P}$ respectively with $\phi_i^2= q^n+(q^n-1)\phi_i$ for $i=0,1$. In this section we show that $g_0*g_1= T_P(\alpha)$, where $g_i= q^{-n/2}\phi_i$ for $i=0,1$. 
	
	\begin{proposition}\label{pro_36}
		$g_0g_1= T_P(\alpha)$.
	\end{proposition}
	\begin{proof}
		Let us choose $\psi_i \in \mathcal{H}(G,\rho)$ for $i=0,1$ such that $\text{supp}(\psi_i)= \mathfrak{P}w_i\mathfrak{P}$ for $i=0,1$. So $\phi_i$ is a scalar multiple of $\psi_i$ for $i=0,1$. Hence $\phi_i = \lambda_i \psi_i$ where $\lambda_i \in \mathbb{C}^{\times}$ for $i=0,1$. Let $\psi_i(w_i)=A \in \mathrm{Hom}_{\mathfrak{P}\cap ^{w_i}\mathfrak{P}}(^{w_i}\rho^{\vee},\rho^{\vee})$ for $i=0,1$ and $W$ be the space of $\rho$. So $A^2= 1_{W^{\vee}}$. From Propostion \ref{prop_35}, we have $(\psi_0*\psi_1)(\zeta)= \psi_0(w_0)\psi_1(w_1)= A^2=1_{W^{\vee}}$. Now let $\psi_i$ satisfies the quadratic relation given by $\psi_i^2 = a\psi_i+b$ where $a,b \in \mathbb{R}$ for $i=0,1$. As 
		$\psi_i^2 = a\psi_i+b \Longrightarrow (-\psi_i)^2 = (-a)(-\psi_i)+b$, so $a$ can be arranged such that $a>0$. We can see that $1 \in {\mathcal{H}(G,\rho)}$ is defined as below:
		
		\[
		1(x)=
		\begin{cases}
		0,   &\text{if $x \notin \mathfrak{P}$;}\\
		\rho^{\vee}(x) &\text{if $x \in \mathfrak{P}$.}
		\end{cases}
		\]
		
		Let us consider $\psi_i^2(1) = \int_G \psi_i(y)\psi_i(y^{-1})dy$ for $i=0,1$. Now let $y=pw_ip'$ where $p,p' \in \mathfrak{P}$ for $i=0,1$. So we have
		
		\begin{align*}
		\psi_i^2(1) &= \int_{\mathfrak{P}w_i\mathfrak{P}}\psi_i(pw_ip')\psi_i(p^{'-1}w_i^{-1}p^{-1})d(pw_ip')\\
		&=\int_{\mathfrak{P}w_i\mathfrak{P}}\rho^{\vee}(p)\psi_i(w_i)\rho^{\vee}(p')\rho^{\vee}(p^{'-1})\psi_i(w_i^{-1})\rho^{\vee}(p^{-1})d(pw_ip')\\
		&=\int_{\mathfrak{P}w_i\mathfrak{P}}\rho^{\vee}(p)\psi_i(w_i)\psi_i(w_i^{-1})\rho^{\vee}(p^{-1})d(pw_ip')\\
		&=\int_{\mathfrak{P}w_i\mathfrak{P}}\rho^{\vee}(p)\psi_i(w_i)\psi_i(w_i)\rho^{\vee}(p^{-1})d(pw_ip')\\
		&=\int_{\mathfrak{P}w_i\mathfrak{P}}\rho^{\vee}(p)A^2\rho^{\vee}(p^{-1})d(pw_ip')\\
		&=\int_{\mathfrak{P}w_i\mathfrak{P}}A^2\rho^{\vee}(p)\rho^{\vee}(p^{-1})d(pw_ip')\\
		&=A^2\text{vol}(\mathfrak{P}w_i\mathfrak{P})\\
		&=1_{W^{\vee}}\text{vol}(\mathfrak{P}w_i\mathfrak{P}).
		\end{align*}\par
		
		So $\psi_i^2(1)=1_{W^{\vee}}\text{vol}(\mathfrak{P}w_i\mathfrak{P})$ for $i=0,1$. We already know that $\psi_i^2 = a\psi_i+b$ where $a,b \in \mathbb{R}$ and for $i=0,1$. Now evaluating the expression $\psi_i^2 = a\psi_i+b$ at 1, we have 
		$\psi_i^2(1) = a\psi_i(1)+b 1(1)$. We can see that $\psi_i(1)=0$ as support of $\psi_i$ is $\mathfrak{P}w_i\mathfrak{P}$ for $i=0,1$. We have seen before that $\psi_i^2(1) = 1_{W^{\vee}}\text{vol}(\mathfrak{P}w_i\mathfrak{P})$ for $i=0,1$ and as $1 \in \mathfrak{P}, 1(1)= \rho^{\vee}(1)=1_{W^{\vee}}$. So $\psi_i^2(1) = a\psi_i(1)+b 1(1) \Longrightarrow 1_{W^{\vee}}\text{vol}(\mathfrak{P}w_i\mathfrak{P})=1_{W^{\vee}}b$ for $i=0,1$. Comparing coefficients of $1_{W^{\vee}}$ on both sides of the equation $1_{W^{\vee}}\text{vol}(\mathfrak{P}w_i\mathfrak{P})=1_{W^{\vee}}b$ for $i=0,1$ we get
		
		\begin{center}
			$b=\text{vol}(\mathfrak{P}w_i\mathfrak{P})$. 
		\end{center}\par

		As $\phi_i= \lambda_i\psi_i$ for $i=0,1$, hence $\phi_i^2= \lambda_i^2\psi_i^2=\lambda_i^2(a\psi_i+b)= (\lambda_i a)(\lambda_i \psi_i)+ \lambda_i^2 b= (\lambda_i a)\phi_i+  \lambda_i^2 b$ for $i=0,1$. But $\phi_i^2= (q^n-1)\phi_i +q^n$ for $i=0,1$. So $\phi_i^2=(\lambda_i a)\phi_i+  \lambda_i^2 b = (q^n-1)\phi_i +q^n$ for $i=0,1$. As $\phi_i$ and $1$ are linearly independent, hence $\lambda_i a= (q^n-1)$ for $i=0,1$. Therefore $\lambda_i= \frac{q^n-1}{a}$ for $i=0,1$. As $a>0, a\in \mathbb{R}$, so $\lambda_i>0, \lambda_i \in \mathbb{R}$ for $i=0,1$. Similarly, as $\phi_i$ and $1$ are linearly independent, hence $\lambda_i^2 b= q^n \Longrightarrow \lambda_i^2= \frac{q^n}{b}$ for $i=0,1$.\par 
		
		Now $\mathfrak{P}w_i\mathfrak{P}= \underset{x \in \mathfrak{P}/\mathfrak{P} \cap ^{w_i}\mathfrak{P}}{\amalg} xw_i\mathfrak{P} \Longrightarrow \text{vol}(\mathfrak{P}w_i\mathfrak{P})= [\mathfrak{P}w_i\mathfrak{P}:\mathfrak{P}]\text{vol}\mathfrak{P}=[\mathfrak{P}w_i\mathfrak{P}:\mathfrak{P}]=[\mathfrak{P}:\mathfrak{P} \cap ^{w_i}\mathfrak{P}]$ for $i=0,1$. Hence $b=\text{vol}(\mathfrak{P}w_i\mathfrak{P})=[\mathfrak{P}:\mathfrak{P} \cap ^{w_i}\mathfrak{P}]$ for $i=0,1$. Now as $\lambda_0^2= \lambda_1^2= \frac{q^n}{b} \Longrightarrow \lambda_0= \lambda_1= \frac{q^{n/2}}{b^{1/2}}=\frac{q^{n/2}}{[\mathfrak{P}:\mathfrak{P} \cap ^{w_0}\mathfrak{P}]^{1/2}}$. Therefore

		\begin{align*}
		\phi_0\phi_1 &=(\lambda_0 \psi_0)(\lambda_1\psi_1)\\
		&=\lambda_0^2\psi_0\psi_1\\
		&= \frac{q^n \psi_0 \psi_1}{[\mathfrak{P}:\mathfrak{P} \cap ^{w_0}\mathfrak{P}]}.
		\end{align*}\par
		
		We have seen before that, $\mathfrak{P}= K_{-,1}\mathfrak{P}_{0}K_{+,0}$ and $\mathfrak{P} \cap ^{w_0}\mathfrak{P}=K_{-,1}\mathfrak{P}_{0}K_{+,1}$. So
		
		\begin{align*}
		[\mathfrak{P}:\mathfrak{P} \cap ^{w_0}\mathfrak{P}]&=|\frac{K_{+,0}}{K_{+,1}}|\\
		&=|\lbrace X \in \mathrm{M}_n(k_E), u \in \mathrm{M}_{n \times 1}(k_E) \mid X+ ^t{}\overline{X}+u{}^t \overline{u}=0 \rbrace|\\
		&=(q^{2n})(q)^{\frac{(n)(n-1)}{2}}\\
		&=(q^{2n})(q^{n^2-n})\\
		&=q^{n^2+n}.
		\end{align*}\par
		
		Hence
		\begin{align*}
		(\phi_0\phi_1)(\zeta )&=\frac{q^n (\psi_0 \psi_1)(\zeta)}{[\mathfrak{P}:\mathfrak{P} \cap ^{w_0}\mathfrak{P}]}\\ 
		&=\frac{q^n (\psi_0 \psi_1)(\zeta)}{ q^{n^2+n}}\\
		&=q^{-n^2}1_{W^{\vee}}.
		\end{align*}\par
		
		Recall $g_i = q^{-n/2}\phi_i$ for $i=0,1$. We know that $\phi_i^2= (q^n-1)\phi_i +q^n$ for $i=0,1$. So for $i=0,1$ we have
		
		\begin{align*}
		g_i^2 &= q^{-n}\phi_i^2\\
		&=q^{-n}((q^n-1)\phi_i +q^n)\\
		&=(1-q^{-n})\phi_i+1\\
		&=(1-q^{-n})q^{n/2}g_i+1\\
		&=(q^{n/2}- q^{-n/2})g_i+1.
		\end{align*}\par
		
		So $g_0g_1=(q^{-n/2}\phi_1)(q^{-n/2}\phi_2)= q^{-n}\phi_1\phi_2 \Longrightarrow (g_0g_1)(\zeta)= q^{-n}(\phi_1\phi_2)(\zeta)= q^{-n}q^{-n^2}1_{W^{\vee}}= q^{-n^2-n}1_{W^{\vee}}$. From the earlier discussion in this section we have $T_P(\alpha)(\zeta)=\delta_P^{1/2}(\zeta)1_{W^{\vee}}$. From section 9.1, we have $\delta_P(\zeta)= q^{-2n^2-2n}$. Hence $\delta_P^{1/2}(\zeta)= q^{-n^2-n}$. Therefore $(g_0g_1)(\zeta)= T_P(\alpha)(\zeta)$. So $(g_0g_1)(\zeta)= T_P(\alpha)(\zeta)$. We have $\text{supp}(T_P(\alpha))= \mathfrak{P}\zeta\mathfrak{P}$. As $\text{supp}(g_i)=\mathfrak{P}w_i\mathfrak{P}$, Lemma \ref{lem_10} gives $\text{supp}(g_0g_1)= \mathfrak{P}\zeta\mathfrak{P}$. Therefore $g_0g_1= T_P(\alpha)$.
	\end{proof}
	
	\subsubsection{Ramified case:}
	We know that $\mathcal{H}(G,\rho)=\langle \phi_0, \phi_1 \rangle$ where $\phi_0$ is supported on $\mathfrak{P}w_0\mathfrak{P}$ and $\phi_1$ is supported on $\mathfrak{P}w_1\mathfrak{P}$ respectively with $\phi_i^2= q^{n/2}+(q^{n/2}-1)\phi_i$ for $i=0,1$. In this section we show that $g_0*g_1= T_P(\alpha)$, where $g_i= q^{-n/4}\phi_i$ for $i=0,1$.
	
	\begin{proposition}\label{pro_38}
		$g_0g_1= T_P(\alpha)$.
	\end{proposition}
	\begin{proof}	
		Let us choose $\psi_i \in \mathcal{H}(G,\rho)$ for $i=0,1$ such that $\text{supp}(\psi_i)= \mathfrak{P}w_i\mathfrak{P}$ for $i=0,1$. So $\phi_i$ is a scalar multiple of $\psi_i$ for $i=0,1$. Hence $\phi_i = \lambda_i \psi_i$ where $\lambda_i \in \mathbb{C}^{\times}$ for $i=0,1$. Let $\psi_i(w_i)=A_i \in \mathrm{Hom}_{\mathfrak{P}\cap ^{w_i}\mathfrak{P}}(^{w_i}\rho^{\vee},\rho^{\vee})$ for $i=0,1$ and $W$ be the space of $\rho$. So $A_i^2= 1_{W^{\vee}}$ for $i=0,1$. From section 5.1 on page 24 in \cite{MR2276353}, we can say that $A_0= A_1$. From Proposition \ref{prop_35}, we have $(\psi_0*\psi_1)(\zeta)= \psi_0(w_0)\psi_1(w_1)= A_0A_1=A_0^2=1_{W^{\vee}}$. Now let $\psi_i$ satisfies the quadratic relation given by $\psi_i^2 = a_i\psi_i+b_i$ where $a_i,b_i \in \mathbb{R}$ for $i=0,1$. As $\psi_i^2 = a_i\psi_i+b_i \Longrightarrow (-\psi_i)^2 = (-a_i)(-\psi_i)+b_i$, so $a_i$ can be arranged such that $a_i>0$ for $i=0,1$. We can see that $1 \in {\mathcal{H}(G,\rho)}$ is defined as below:
		
		\[
		1(x)=
		\begin{cases}
		0,   &\text{if $x \notin \mathfrak{P}$;}\\
		\rho^{\vee}(x) &\text{if $x \in \mathfrak{P}$.}
		\end{cases}
		\]
		
		Let us consider $\psi_0^2(1) = \int_G \psi_0(y)\psi_0(y^{-1})dy$. Now let $y=pw_0p'$ where $p,p' \in \mathfrak{P}$. So we have
		
		\begin{align*}
		\psi_0^2(1) &= \int_{\mathfrak{P}w_0\mathfrak{P}}\psi_0(pw_0p')\psi_0(p'^{-1}w_0^{-1}p^{-1})d(pw_0p')\\
		&=\int_{\mathfrak{P}w_0\mathfrak{P}}\rho^{\vee}(p)\psi_0(w_0)\rho^{\vee}(p')\rho^{\vee}(p^{'-1})\psi_0(w_0^{-1})\rho^{\vee}(p^{-1})d(pw_0p')\\
		&=\int_{\mathfrak{P}w_0\mathfrak{P}}\rho^{\vee}(p)\psi_0(w_0)\psi_0(w_0^{-1})\rho^{\vee}(p^{-1})d(pw_0p')\\
		&=\int_{\mathfrak{P}w_0\mathfrak{P}}\rho^{\vee}(p)\psi_0(w_0)\psi_0(w_0)\rho^{\vee}(p^{-1})d(pw_0p')\\
		&=\int_{\mathfrak{P}w_0\mathfrak{P}}\rho^{\vee}(p)A_0^2\rho^{\vee}(p^{-1})d(pw_0p')\\
		&=\int_{\mathfrak{P}w_0\mathfrak{P}}A_0^2\rho^{\vee}(p)\rho^{\vee}(p^{-1})d(pw_0p')\\
		&=A_0^2\text{vol}(\mathfrak{P}w_0\mathfrak{P})\\
		&=1_{W^{\vee}}\text{vol}(\mathfrak{P}w_0\mathfrak{P}).
		\end{align*}\par
		
		So $\psi_0^2(1)=1_{W^{\vee}}\text{vol}(\mathfrak{P}w_0\mathfrak{P})$. We already know that 
		$\psi_0^2 = a_0\psi_0+b_0$ where $a_0,b_0 \in \mathbb{R}$. Now evaluating the expression $\psi_0^2 = a_0 \psi_0+b_0$ at 1, we have 
		$\psi_0^2(1) = a_0\psi_0(1)+b_0 1(1)$. We can see that $\psi_0(1)=0$ as support of $\psi_0$ is $\mathfrak{P}w_0\mathfrak{P}$. We have seen before that $\psi_0^2(1) = 1_{W^{\vee}}\text{vol}(\mathfrak{P}w_0\mathfrak{P})$ and as $1 \in \mathfrak{P}, 1(1)= \rho^{\vee}(1)= 1_{W^{\vee}}$. So $\psi_0^2(1) = a_0\psi_i(1)+b_0 1(1) \Longrightarrow 1_{W^{\vee}}\text{vol}(\mathfrak{P}w_0\mathfrak{P})=1_{W^{\vee}}b_0$. Comparing  coefficients of $1_{W^{\vee}}$ on both sides of the equation $1_{W^{\vee}}b_0=1_{W^{\vee}}\text{vol}(\mathfrak{P}w_0\mathfrak{P})$ we get
		
		\begin{center}
			$b_0=\text{vol}(\mathfrak{P}w_0\mathfrak{P})$.
		\end{center}\par
		
		As $\phi_0= \lambda_0\psi_0$, hence $\phi_0^2= \lambda_0^2\psi_0^2=\lambda_0^2(a_0\psi_0+b_0)= (\lambda_0 a_0)(\lambda_0 \psi_0)+ \lambda_0^2 b_0= (\lambda_0 a_0)\phi_0+  \lambda_0^2 b_0$. But $\phi_0^2= (q^{n/2}-1)\phi_0 +q^{n/2}$. So $\phi_0^2=(\lambda_0 a_0)\phi_0+  \lambda_0^2 b_0= (q^{n/2}-1)\phi_0 +q^{n/2}$. As $\phi_0$ and $1$ are linearly independent, hence $\lambda_0 a_0= (q^{n/2}-1)$. Therefore $\lambda_0= \frac{q^{n/2}-1}{a_0}$. As $a_0>0, a_0\in \mathbb{R}$, so $\lambda_0>0, \lambda_0 \in \mathbb{R}$. Similarly, as $\phi_0$ and $1$ are linearly independent, hence $\lambda_0^2 b= q^{n/2} \Longrightarrow \lambda_0^2= \frac{q^{n/2}}{b_0}$. \par 
		
		Now $\mathfrak{P}w_0\mathfrak{P}= \underset{x \in \mathfrak{P}/\mathfrak{P} \cap ^{w_0}\mathfrak{P}}{\amalg} xw_0\mathfrak{P} \Longrightarrow \text{vol}(\mathfrak{P}w_0\mathfrak{P})= [\mathfrak{P}w_0\mathfrak{P}:\mathfrak{P}]\text{vol}\mathfrak{P}=[\mathfrak{P}w_0\mathfrak{P}:\mathfrak{P}]=[\mathfrak{P}:\mathfrak{P} \cap ^{w_0}\mathfrak{P}]$. Hence $b_0=\text{vol}(\mathfrak{P}w_0\mathfrak{P})=[\mathfrak{P}:\mathfrak{P} \cap ^{w_0}\mathfrak{P}]$. Now as $\lambda_0^2= \frac{q^{n/2}}{b_0} \Longrightarrow \lambda_0= \frac{q^{n/4}}{b_0^{1/2}}=\frac{q^{n/4}}{[\mathfrak{P}:\mathfrak{P} \cap ^{w_0}\mathfrak{P}]^{1/2}}$.\par
		
		We have seen before that, $\mathfrak{P}= K_{-,1}\mathfrak{P}_{0}K_{+,0}$ and $\mathfrak{P} \cap ^{w_0}\mathfrak{P}=K_{-,1}\mathfrak{P}_{0}K_{+,1}$. So
		
		\begin{align*}
		[\mathfrak{P}:\mathfrak{P} \cap ^{w_0}\mathfrak{P}]&=|\frac{K_{+,0}}{K_{+,1}}|\\
		&=|\lbrace X \in \mathrm{M}_n(k_E), u \in \mathrm{M}_{n \times 1}(k_E) \mid X+ ^t{}\overline{X}+u{}^t \overline{u}=0 \rbrace|\\
		&=(q^n)(q^{\frac{(n)(n-1)}{2}})\\
		&=q^{\frac{n^2+n}{2}}.
		\end{align*}\par

		So
		\begin{center}
			$\lambda_0=\frac{q^{n/4}}{[\mathfrak{P}:\mathfrak{P} \cap ^{w_0}\mathfrak{P}]^{1/2}}=\frac{q^{n/4}}{q^{\frac{n^2+n}{4}}}.$
		\end{center}\par
		
		Let us consider $\psi_1^2(1) = \int_G \psi_1(y)\psi_1(y^{-1})dy$. Now let $y=pw_1p'$ where $p,p' \in \mathfrak{P}$. So we have
		
		\begin{align*}
		\psi_1^2(1) &= \int_{\mathfrak{P}w_1\mathfrak{P}}\psi_1(pw_1p')\psi_1(p^{'-1}w_1^{-1}p^{-1})d(pw_1p')\\
		&=\int_{\mathfrak{P}w_1\mathfrak{P}}\rho^{\vee}(p)\psi_1(w_1)\rho^{\vee}(p')\rho^{\vee}(p^{'-1})\psi_1(w_1^{-1})\rho^{\vee}(p^{-1})d(pw_1p')\\
		&=\int_{\mathfrak{P}w_1\mathfrak{P}}\rho^{\vee}(p)\psi_1(w_1)\psi_1(w_1^{-1})\rho^{\vee}(p^{-1})d(pw_1p')\\
		&=\int_{\mathfrak{P}w_1\mathfrak{P}}\rho^{\vee}(p)\psi_1(w_1)\psi_1(-w_1)\rho^{\vee}(p^{-1})d(pw_1p')\\
		&=\int_{\mathfrak{P}w_1\mathfrak{P}}\rho^{\vee}(p)\psi_1(w_1)\rho^{\vee}(-1)\psi_1(w_1)\rho^{\vee}(p^{-1})d(pw_1p')\\
		&=\rho^{\vee}(-1)\int_{\mathfrak{P}w_1\mathfrak{P}}A_1^2\rho^{\vee}(p)\rho^{\vee}(p^{-1})d(pw_1p')\\
		&=\rho^{\vee}(-1)A_1^2\text{vol}(\mathfrak{P}w_1\mathfrak{P})\\
		&=\rho^{\vee}(-1)1_{W^{\vee}}\text{vol}(\mathfrak{P}w_1\mathfrak{P}).
		\end{align*}\par
		
		So $\psi_1^2(1)=1_{W^{\vee}}\text{vol}(\mathfrak{P}w_1\mathfrak{P})$. We already know that 
		$\psi_1^2 = a_1\psi_1+b_1$ where $a_1,b_1 \in \mathbb{R}$. Now evaluating the expression $\psi_1^2 = a_1 \psi_1+b_1$ at 1, we have 
		$\psi_1^2(1) = a_1\psi_1(1)+b_1 1(1)$. We can see that $\psi_1(1)=0$ as support of $\psi_1$ is $\mathfrak{P}w_1\mathfrak{P}$. We have seen before that $\psi_1^2(1) = 1_{W^{\vee}}\text{vol}(\mathfrak{P}w_1\mathfrak{P})$ and as $1 \in \mathfrak{P}, 1(1)= \rho^{\vee}(1)= 1_{W^{\vee}}$. So $\psi_1^2(1) = a_1\psi_i(1)+b_1 1(1) \Longrightarrow \rho^{\vee}(-1)1_{W^{\vee}}\text{vol}(\mathfrak{P}w_1\mathfrak{P})=1_{W^{\vee}}b_1$. Comparing coefficients of $1_{W^{\vee}}$ on both sides of the equation $1_{W^{\vee}}b_1=1_{W^{\vee}}\rho^{\vee}(-1)\text{vol}(\mathfrak{P}w_1\mathfrak{P})$ we get
		
		\begin{center}
			$b_1=\rho^{\vee}(-1)\text{vol}(\mathfrak{P}w_1\mathfrak{P})$.\par 
		\end{center}\par
		
		As $\phi_1= \lambda_1\psi_1$, hence $\phi_1^2= \lambda_1^2\psi_1^2=\lambda_1^2(a_1\psi_1+b_1)= (\lambda_1 a_1)(\lambda_1 \psi_1)+ \lambda_1^2 b_1 = (\lambda_0 a_1)\phi_1+  \lambda_1^2 b_1$. But $\phi_1^2= (q^{n/2}-1)\phi_1 +q^{n/2}$. So $\phi_1^2=(\lambda_1 a_1)\phi_1+  \lambda_1^2 b_1= (q^{n/2}-1)\phi_1 +q^{n/2}$. As $\phi_1$ and $1$ are linearly independent, hence $\lambda_1 a_1= (q^{n/2}-1)$. Therefore $\lambda_1= \frac{q^{n/2}-1}{a_1}$. As $a_1>0, a_1\in \mathbb{R}$, so $\lambda_1>0, \lambda_1 \in \mathbb{R}$. Similarly, as $\phi_1$ and $1$ are linearly independent, hence $\lambda_1^2 b= q^{n/2} \Longrightarrow \lambda_1^2= \frac{q^{n/2}}{b_1}$. \par 
		
		Now $\mathfrak{P}w_1\mathfrak{P}= \underset{x \in \mathfrak{P}/\mathfrak{P} \cap ^{w_1}\mathfrak{P}}{\amalg} xw_1\mathfrak{P} \Longrightarrow \text{vol}(\mathfrak{P}w_1\mathfrak{P})= [\mathfrak{P}w_1\mathfrak{P}:\mathfrak{P}]\text{vol}\mathfrak{P}=[\mathfrak{P}w_1\mathfrak{P}:\mathfrak{P}]=[\mathfrak{P}:\mathfrak{P} \cap ^{w_1}\mathfrak{P}]$. Hence $b_1=\text{vol}(\mathfrak{P}w_1\mathfrak{P})=[\mathfrak{P}:\mathfrak{P} \cap ^{w_1}\mathfrak{P}]$. Now as $\lambda_1^2= \frac{q^{n/2}}{b_1} \Longrightarrow \lambda_1= \frac{q^{n/4}}{b_1^{1/2}}=\frac{q^{n/4}}{[\mathfrak{P}:\mathfrak{P} \cap ^{w_1}\mathfrak{P}]^{1/2}}$.\par
		
		We have seen before that $\mathfrak{P}= K_{-,1} \mathfrak{P}_0 K_{+,0}, ^{w_1}\mathfrak{P}= K_{-,2} \mathfrak{P}_0 K_{+,-1}$. So $\mathfrak{P} \cap ^{w_1}\mathfrak{P}= K_{-,2} \mathfrak{P}_0K_{+,0}$. Hence
		
		\begin{align*}
		[\mathfrak{P}:\mathfrak{P} \cap ^{w_1}\mathfrak{P}]&=|\frac{K_{-,1}}{K_{-,2}}|\\
		&=|\lbrace X \in \mathrm{M}_n(k_E), u \in \mathrm{M}_{n \times 1}(k_E) \mid X- ^t{}\overline{X}-u{}^t \overline{u}=0 \rbrace|\\
		&=q^{\frac{n^2+n}{2}}.
		\end{align*}\par
		
		So
		\begin{center}
			$\lambda_1=\frac{q^{n/4}}{[\mathfrak{P}:\mathfrak{P} \cap ^{w_1}\mathfrak{P}]^{1/2}}=\frac{q^{n/4}}{q^{\frac{n^2+n}{4}}(\rho(-1))^{1/2}}.$
		\end{center}\par
		
		Hence
		\begin{align*}
		(\phi_0\phi_1)(\zeta)&=(\lambda_0\psi_0)(\lambda_1\psi_1)(\zeta)\\
		&=(\lambda_0\lambda_1)(\psi_0\psi_1)(\zeta)\\
		&=\frac{q^{n/4}}{q^{\frac{n^2+n}{4}}}\frac{q^{n/4}}{q^{\frac{n^2+n}{4}}(\rho(-1))^{1/2}}1_W^{\vee}\\
		&=\frac{q^\frac{-n^2}{2}1_W^{\vee}}{(\rho(-1))^{1/2}}.
		\end{align*}\par
		
		As $-1 \in Z(\mathfrak{P})$ and $\rho^{\vee}$ is a representation of $\mathfrak{P}$, so $\rho^{\vee}(-1)= \omega_{\rho^{\vee}}(-1)$ where $\omega_{\rho^{\vee}}$ is the central character of $\mathfrak{P}$. Now $1=\omega_{\rho^{\vee}}(1)=(\omega_{\rho^{\vee}}(-1))^2$, so 
		$\rho^{\vee}(-1)= \omega_{\rho^{\vee}}(-1)= \pm 1$. We have seen before that $\lambda_1= \frac{q^{n/2}-1}{a_1}$ and $a_1 \in \mathbb{R}, a_1>0$, so $\lambda_1 >0$. But we know that $\lambda_1=\frac{q^{n/4}}{[\mathfrak{P}:\mathfrak{P} \cap ^{w_1}\mathfrak{P}]^{1/2}}=\frac{q^{n/4}}{q^{\frac{n^2+n}{4}}(\rho(-1))^{1/2}}$, hence $\rho^{\vee}(-1)=1$.\par
		
		Recall $g_i = q^{-n/4}\phi_i$ for $i=0,1$. We know that $\phi_i^2= (q^{n/2}-1)\phi_i +q^{n/2}$ for $i=0,1$. So for $i=0,1$ we have
		
		\begin{align*}
		g_i^2 &= q^{-n/2}\phi_i^2\\
		&=q^{-n/2}((q^{n/2}-1)\phi_i +q^{n/2})\\
		&=(1-q^{-n/2})\phi_i+1\\
		&=(1-q^{-n/2})q^{n/4}g_i+1\\
		&=(q^{n/4}- q^{-n/4})g_i+1.
		\end{align*}\par
		
		So $g_0g_1= q^{-n/2}\phi_1\phi_2 \Longrightarrow (g_0g_1)(\zeta)=q^{-n/2}(\phi_0\phi_1)(\zeta)= q^{-n/2}\frac{q^\frac{-n^2}{2}1_W^{\vee}}{(\rho(-1))^{1/2}}=q^\frac{-n^2-n}{2}1_W^{\vee}$. From the earlier discussion in this section we have $T_P(\alpha)(\zeta)=\delta_P^{1/2}(\zeta)1_{W^{\vee}}$. From section 9.1, we have $\delta_P(\zeta)=  q^{-n^2-n}$. Hence $\delta_P^{1/2}(\zeta)= q^{\frac{-n^2-n}{2}}$. Therefore $(g_0g_1)(\zeta)= T_P(\alpha)(\zeta)$. So $(g_0g_1)(\zeta)= T_P(\alpha)(\zeta)$. We have $\text{supp}(T_P(\alpha))= \mathfrak{P}\zeta\mathfrak{P}$. As $\text{supp}(g_i)=\mathfrak{P}w_i\mathfrak{P}$, Lemma \ref{lem_10}  gives $\text{supp}(g_0g_1)= \mathfrak{P}\zeta\mathfrak{P}$. Therefore $g_0g_1= T_P(\alpha)$.
	\end{proof}
			
\subsection{Calculation of $m_L(\pi\nu)$}
			
Recall that $\pi= \lambda \chi$ where $\lambda$ is an irreducible supercuspidal depth zero representation of $\mathrm{GL}_n(E)$ and $\chi$ is a supercuspidal depthzero character of $\mathrm{U}_1(E)$. Note that $\pi\nu$ lies in $\mathfrak{R}^{[L,\pi]_L}(L)$. Recall $m_L$ is an equivalence of categories. As $\pi\nu$ is an irreducible representation of $L$, it follows that $m_L(\pi\nu)$ is a simple $\mathcal{H}(L,\rho_0)$-module. In this section, we identify the simple $\mathcal{H}(L,\rho_0)$-module corresponding to $m_L(\pi\nu)$. Calculating $m_L(\pi\nu)$ will be useful in answering the question in next section.\par

From section 2.5, we know that $\pi =c$-$Ind_{\widetilde{\mathfrak{P}_0}}^L \widetilde{\rho_0}$, where $\widetilde{\mathfrak{P}_0}=\langle \zeta \rangle \mathfrak{P}_0, \widetilde{\rho_0}(\zeta^k j)=\rho_0(j)$ for $j \in \mathfrak{P}_0, k \in \mathbb{Z}$. Let us recall that $\nu$ is unramified character of $L$ from section 1. Let $V$ be space of $\pi \nu$ and $W$ be space of $\rho_0$. Recall $m_L(\pi\nu)= \mathrm{Hom}_{\mathfrak{P}_0}(\rho_0,\pi \nu)$. Let $f \in \mathrm{Hom}_{\mathfrak{P}_0}(\rho_0,\pi \nu)$. As $\mathfrak{P}_0$ is a compact open subgroup of $L$ and $\nu$ is an unramified character of $L$, so $\nu(j)=1$ for $j \in \mathfrak{P}_0$. We already know that $\alpha \in \mathcal{H}(L,\rho_0)$ with support of $\alpha$ being $\mathfrak{P}_0\zeta$ and $\alpha(\zeta)= 1_{W^{\vee}}$. Let $w \in W$ and we have seen in section 2.4 that the way  $\mathcal{H}(L,\rho_0)$ acts on $\mathrm{Hom}_{\mathfrak{P}_0}(\rho_0,\pi \nu)$ is given by:

\begin{align*}
(\alpha.f)(w)&= \int_L(\pi \nu)(l)f(\alpha^{\vee}(l^{-1})w)dl\\
&=\int_L(\pi \nu)(l)f((\alpha(l))^{\vee}w)dl\\
&=\int_{\mathfrak{P}_0}(\pi \nu)(p\zeta)f((\alpha(p\zeta))^{\vee}w)dp\\
&=\int_{\mathfrak{P}_0}(\pi \nu)(p\zeta)f((\rho_0^{\vee}(p)\alpha(\zeta))^{\vee}w)dp\\
&=\int_{\mathfrak{P}_0}(\pi \nu)(p\zeta)f((\rho_0^{\vee}(p)1_{W^{\vee}})^{\vee}w)dp\\
&=\int_{\mathfrak{P}_0}(\pi \nu)(p\zeta)f((\rho_0^{\vee}(p))^{\vee}w)dp\\
&=\int_{\mathfrak{P}_0}\pi(p\zeta) \nu(p\zeta)f((\rho_0^{\vee}(p))^{\vee}w)dp\\
&=\int_{\mathfrak{P}_0}\pi(p\zeta) \nu(\zeta)f((\rho_0^{\vee}(p))^{\vee}w)dp.\\
\end{align*}\par

Now $\langle, \rangle \colon W \times W^{\vee} \longrightarrow \mathbb{C}$ is given by: $\langle w, \rho_0^{\vee}(p)w^{\vee}\rangle = \langle \rho_0(p^{-1})w, w^{\vee} \rangle$ for $p \in \mathfrak{P}_0, w \in W$. So we have $(\rho_0^{\vee}(p))^{\vee}=\rho_0(p^{-1})$ for $p \in \mathfrak{P}_0$. Hence

\begin{align*}
(\alpha.f)(w)&=\int_{\mathfrak{P}_0}\pi(p\zeta) \nu(\zeta)f(\rho_0(p^{-1})w) dp.
\end{align*}\par

As $f \in \mathrm{Hom}_{\mathfrak{P}_0}(\rho_0,\pi \nu)$, so $(\pi \nu)(p)f(w) = f(\rho_0(p)w)$ for $p \in \mathfrak{P}_0, w \in W$. Hence

\begin{align*}
(\alpha.f)(w)&=\nu(\zeta)\int_{\mathfrak{P}_0}\pi(p\zeta)(\pi \nu)(p^{-1})f(w)dp\\
&=\nu(\zeta)\int_{\mathfrak{P}_0}\pi(p\zeta)\pi(p^{-1})\nu(p^{-1})f(w)dp\\
&=\nu(\zeta)\int_{\mathfrak{P}_0}\pi(p\zeta)\pi(p^{-1})f(w)dp.
\end{align*}\par

Now as $\pi = c$-$Ind_{\widetilde{\mathfrak{P}_0}}^L \widetilde\rho_0$ and $\widetilde{\mathfrak{P}_0}=\langle \zeta \rangle \mathfrak{P}_0, \widetilde{\rho_0}(\zeta^k j)=\rho_0(j) $ for $j \in \mathfrak{P}_0, k \in \mathbb{Z}$, so $\pi(p\zeta)=\pi(p)\widetilde{\rho_0}(\zeta)=\pi(p)\rho_0(1)=\pi(p)1_{W^{\vee}}$. Therefore

\begin{align*}
(\alpha.f)(w)&=\nu(\zeta)\int_{\mathfrak{P}_0}\pi(p)\pi(p^{-1})f(w)dp\\
&=\nu(\zeta)f(w)\text{Vol}(\mathfrak{P}_0)\\
&=\nu(\zeta)f(w)
\end{align*}

So $(\alpha.f)(w)=\nu(\zeta)f(w)$ for $w \in W$. So $\alpha$ acts on $f$ by multiplication by $\nu(\zeta)$. Recall for $\lambda \in \mathbb{C}^{\times}$, we write $\mathbb{C}_{\lambda}$ for the $\mathcal{H}(L,\rho_0)$-module with underlying abelian group $\mathbb{C}$ such that $\alpha.z=\lambda z$ for $z \in \mathbb{C}_{\lambda}$. Therefore $m_L(\pi \nu) \cong \mathbb{C}_{\nu(\zeta)}$.\par

\section{Answering the question}
			
			Recall the following commutative diagram which we have described earlier.
			
			\[
			\begin{CD}
			\mathfrak{R}^{[L,\pi]_G}(G)    @>m_G>>    \mathcal{H}(G,\rho)-Mod\\
			@A\iota_P^GAA                                    @A(T_P)_*AA\\
			\mathfrak{R}^{[L,\pi]_L}(L)    @>m_L>>     \mathcal{H}(L,\rho_0)-Mod
			\end{CD}
			\tag{CD}\]\par
			
			Recall that in the unramified case when $n$ is even or in the ramified case when $n$ is odd we have  $N_G(\rho_0)= Z(L) \mathfrak{P}_0$. Thus $\mathfrak{I}_G(\rho)= \mathfrak{P}(Z(L)\mathfrak{P}_0)\mathfrak{P}= \mathfrak{P}Z(L)\mathfrak{P}$.\par
			
			From Corollary 6.5 in \cite{MR1486141} it follows that if $\mathfrak{I}_G(\rho) \subseteq \mathfrak{P}L\mathfrak{P}$ then \[T_P \colon \mathcal{H}(L,\rho_0) \longrightarrow \mathcal{H}(G,\rho)\] is an isomorphism of $\mathbb{C}$-algebras. As we have $\mathfrak{I}_G(\rho)=\mathfrak{P}Z(L) \mathfrak{P}$ in the unramified case when $n$ is even or in the ramified case when $n$ is odd, so $\mathcal{H}(L,\rho_0) \cong \mathcal{H}(G,\rho)$ as $\mathbb{C}$-algebras. So from the commutative diagram (CD), we can conclude that $\iota_P^G(\pi \nu)$ is irreducible for any unramified character $\nu$ of $L$.
			
        	Recall that $\pi\nu$ lies in $\mathfrak{R}^{[L,\pi]_L}(L)$. Note that from the above commutative diagram, it follows that $\iota_P^G(\pi\nu)$ lies in $\mathfrak{R}^{[L,\pi]_G}(G)$ and $m_G(\iota_P^G(\pi\nu))$ is an $\mathcal{H}(G,\rho)$-module. Recall $m_L(\pi \nu) \cong \mathbb{C}_{\nu(\zeta)}$ as $\mathcal{H}(L,\rho_0)$-modules. From the commutative diagram (CD), we have $m_G(\iota_P^G(\pi\nu)) \cong (T_P)_*(\mathbb{C}_{\nu(\zeta)})$ as $\mathcal{H}(G,\rho)$-modules. Thus to determine the unramified characters $\nu$ for which $\iota_P^G(\pi\nu)$ is irreducible, we have to understand when $(T_P)_*(\mathbb{C}_{\nu(\zeta)})$ is a simple $\mathcal{H}(G,\rho)$-module.\par

			Using notation on page 438 in \cite{MR2531913}, we have  $\gamma_1=\gamma_2= q^{n/2}$ for unramified case when $n$ is odd and $\gamma_1=\gamma_2= q^{n/4}$ for ramified case when $n$ is even. As in Proposition 1.6 of \cite{MR2531913}, let $\Gamma=\{\gamma_1\gamma_2,-\gamma_1\gamma_2^{-1}, -\gamma_1^{-1}\gamma_2,(\gamma_1\gamma_2)^{-1}\}$. So by Proposition 1.6 in \cite{MR2531913}, 
			$(T_P)_*(\mathbb{C}_{\nu(\zeta)})$ is a simple $\mathcal{H}(G,\rho)$-module $\Longleftrightarrow \nu(\zeta) \notin \Gamma$. Recall $\pi= c$-$Ind_{Z(L)\mathfrak{P}_0}^L \widetilde{\rho_0}$ where $\widetilde{\rho_0}(\zeta^k j)=\rho_0(j)$ for $j \in \mathfrak{P}_0, k \in \mathbb{Z}$ and  $\rho_0=\tau_{\theta}$ for some regular character $\theta$ of $l^{\times}$ with $[l:k_E]=n$. Hence we can conclude that $\iota_P^G(\pi \nu)$ is irreducible for the unramified case when $n$ is odd $\Longleftrightarrow \nu(\zeta) \notin \{q^n,q^{-n},-1\}$, $\theta^{q^{n+1}}= \theta^{-q}$ and  $\iota_P^G(\pi \nu)$ is irreducible for the ramified case when $n$ is even $\Longleftrightarrow \nu(\zeta) \notin \{q^{n/2},q^{-n/2},-1\}$, $\theta^{q^{n/2}}= \theta^{-1}$. That proves Theorem \ref{the_1}.

\nocite{MR1235019}
\nocite{MR1486141}
\nocite{MR1643417}
\nocite{MR2276353}
\nocite{MR2531913}
\nocite{carter_1992}
\nocite{Thomas2014}
\nocite{green}
\nocite{MR1266747}
\nocite{digne_francois_michel_1991}
\nocite{sandeep}
\nocite{heiermann_2011}
\nocite{heiermann_2017}

\end{document}